\documentclass[reqno]{amsart}
\usepackage{amssymb}

\usepackage{tikz}

\usepackage[colorlinks=true]{hyperref}
\hypersetup{urlcolor=blue, citecolor=red}

\newtheorem{Theorem}{Theorem}[section]
\newtheorem{Corollary}[Theorem]{Corollary}
\newtheorem{Lemma}[Theorem]{Lemma}
\newtheorem{Proposition}[Theorem]{Proposition}
\newtheorem{Remark}[Theorem]{Remark}
\numberwithin{equation}{section}

\def\C {\mathbb C}

\def\R {\mathbb R}

\newcommand{\supp}{\operatorname{supp}}
\newcommand{\<}{\langle}
\renewcommand{\>}{\rangle}
\renewcommand{\(}{\left(}
\renewcommand{\)}{\right)}

\newcommand{\lesim}{\lesssim}
\renewcommand{\div}{\operatorname{div}}
\newcommand{\de}[2]{\frac{\partial #1}{\partial #2}}
\newcommand{\p}{\partial}
\newcommand{\Vol}{\operatorname{Vol}}
\newcommand{\sgn}{\operatorname{sgn}}
\newcommand{\id}{\operatorname{Id}}
\newcommand{\dist}{\operatorname{dist}}

\begin{document}
\title[Reconstruction in the partial data Calder\'on problem]{Reconstruction in the partial data Calder\'on problem on admissible manifolds}

\author[Yernat M. Assylbekov]{Yernat M. Assylbekov}
\address{Department of Mathematics, University of Washington, Seattle, WA 98195-4350, USA}
\email{y\_assylbekov@yahoo.com}

\maketitle

\begin{abstract}
We consider the problem of developing a method to reconstruct a potential $q$ from the partial data Dirichlet-to-Neumann map for the Schr\"odinger equation $(-\Delta_g+q)u=0$ on a fixed admissible manifold $(M,g)$. If the part of the boundary that is inaccessible for measurements satisfies a flatness condition in one direction, then we reconstruct the local attenuated geodesic ray transform of the one-dimensional Fourier transform of the potential $q$. This allows us to reconstruct $q$ locally, if the local (unattenuated) geodesic ray transform is constructively invertible. We also reconstruct $q$ globally, if $M$ satisfies certain concavity condition and if the global geodesic ray transform can be inverted constructively. These are reconstruction procedures for the corresponding uniqueness results given by Kenig and Salo \cite{KSa}. Moreover, the global reconstruction extends and improves the constructive proof of Nachman and Street \cite{NaS} in Euclidean setting. We derive a certain boundary integral equation which involves the given partial data and describes the traces of complex geometrical optics solutions. For construction of complex geometrical optics solutions, following \cite{NaS} and improving their arguments, we use a new family of Green's functions for the Laplace-Beltrami operator and the corresponding single layer potentials. The constructive inversion problem for local or global geodesic ray transforms is one of the major topics of interest in integral geometry.
\end{abstract}

\section{Introduction}
In 1980, Alberto Calder\'on \cite{Cal} proposed the problem whether one can determine the electrical conductivity of a medium from voltage and current measurements at the boundary. In the mathematical literature, this problem is known as Calder\'on's inverse conductivity problem.
\medskip

The Calder\'on's problem can be reduced to the problem of determining electric potential $q$ from the Dirichlet-to-Neumann map associated to the Schr\"odinger operator $-\Delta+q$. We will first discuss the case of Euclidean space in dimension $n\ge 3$. In the fundamental paper by Sylvester and Uhlmann \cite{SyU} it was shown that bounded potential in a bounded domain of Euclidean space can be uniquely determined from the knowledge of the Dirichlet-to-Neumann map. Since then, substantial progress has been achieved on Calder\'on's problem. Then corresponding reconstruction procedure was given by Nachman \cite{Na} and independently by Novikov \cite{No}. The reader is referred to recent expository paper by Uhlmann \cite{U} for a survey of progress made on Calder\'on's problem.
\medskip

In the current paper we are interested in the case when the Dirichlet-to-Neumann map is known only on part of the boundary. Let $\Gamma_+$ and $\Gamma_-$ be the open subsets of the boundary where Dirichlet data inputs are prescribed and Neumann data measurements are made. The first result is due to Bukhgeim and Uhlmann \cite{BU}. They prove unique determination result if $\Gamma_+$ and $\Gamma_-$ are roughly complementary and slightly more than half of the boundary. This result has been improved significantly by Kenig, Sj\"ostrand and Uhlmann \cite{KSU} where they show that bounded potential can be uniquely recovered if $\Gamma_-$ possibly very small open subset of the boundary, but $\Gamma_+$ must be slightly larger than the complement of $\Gamma_-$ in the boundary. Constructive proof of this result is given by Nachman and Street \cite{NaS}. For recent results on Calder\'on's inverse problem with partial data, see \cite{KSa-survey}. The approaches of \cite{BU,KSU,NaS} are based on Carleman estimates with boundary terms. 
\medskip

There is a result by Isakov \cite{Is} where he gives uniqueness result when $\Gamma_-=\Gamma_+=\Gamma$ and the inaccessible part of the boundary for measurements is either part of a hyperplane or part of a sphere. This work is based on a reflection argument.
\medskip

In the current paper we consider partial data Calder\'on's problem on manifolds. The methods of \cite{KSU,Is} were unified and extended to so-called admissible manifolds (which will be described below) by Kenig and Salo \cite{KSa} obtaining improved results. To appreciate these improvements, the reader is referred to \cite[Section~3]{KSa} for detailed corresponding discussions. The goal of this paper, is to give the reconstruction procedures to the corresponding results of \cite{KSa}.
\medskip

Let us give the precise mathematical formulation of the problem. Let $(M,g)$ be a compact oriented Riemannian manifold with boundary. Following Bukhgeim and Uhlmann \cite{BU}, we work with the following Hilbert space which is the largest domain of the Laplace-Beltrami operator $\Delta_g$:
$$
H_{\Delta_g}(M)=\{u\in L^2(M):\Delta_g u\in L^2(M)\}.
$$
The trace maps $\mathtt{tr}(u)=u|_{\p M}$ and $\mathtt{tr}_\nu (u)=\de{u}{\nu}\big|_{\p M}$ defined on $C^\infty(M)$ have extensions to a bounded operators $H_{\Delta_g}(M)\to H^{-1/2}(\p M)$ and $H_{\Delta_g}(M)\to H^{-3/2}(\p M)$, respectively; see Proposition~\ref{extensions of the trace operators}.
\medskip

Now we introduce the following space on the boundary $\p M$:
$$
\mathcal H_g(\p M)=\{\mathtt{tr}(u):u\in H_{\Delta_g}(M)\}\subset H^{-1/2}(\p M).
$$
The topology on $\mathcal H_g(\p M)$ is defined in Section~\ref{S2}, right before Proposition~\ref{solvability of the Dirichlet problem}. Under this topology, the operator $\mathtt{tr}:H_{\Delta_g}(M)\to \mathcal H_g(\p M)$ is bounded.
\medskip

We consider the closely related problem for Schr\"odinger operators. Suppose that $q\in L^\infty(M)$ such that $0$ is not a Dirichlet eigenvalue of $-\Delta_g+q$ in $M$. Then for $f\in\mathcal H_g(\p M)$, the Dirichlet problem has a unique solution $u\in H_{\Delta_g}(M)$
\begin{align*}
(-\Delta_g+q)u&=0\quad\text{\rm in}\quad M,\\
\mathtt{tr}(u)&=f\quad\text{\rm on}\quad \p M.
\end{align*}
The Dirichlet-to-Neumann map is defined by
$$
\Lambda_{g,q}(f)=\mathtt{tr}_\nu(u).
$$
By the results of Section~\ref{S2}, $\Lambda_{g,q}$ is a bounded linear operator $\Lambda_{g,q}:\mathcal H_g(\p M)\to H^{-3/2}(\p M)$. Given two open subsets $\Gamma_-,\Gamma_+\subset \p M$. The partial data inverse problem is to determine $q$ from the knowledge of $\Lambda_{g,q}f$ on ${\Gamma_-}$ for all $f\in \mathcal H_g(\p M)$ supported in $\Gamma_+$.
\medskip

We need to introduce the notion of admissible manifolds. 
\medskip

{\bf Definition.} A compact Riemannian manifold $(M,g)$ with boundary of dimension $n\ge 3$, is said to be \emph{admissible} if it is conformal to a submanifold with boundary of $\R\times (M_0,g_0)$ where $(M_0,g_0)$ is simple $(n-1)$-dimensional manifold. By simplicity of $(M_0,g_0)$ we mean that the boundary $\p M_0$ is strictly convex, and for any point $x\in M_0$ the exponential map $\exp_x$ is a diffeomorphism from its maximal domain in $T_x M_0$ onto $M_0$.
\medskip

Compact submanifolds of Euclidean space, the sphere minus a point and of hyperbolic space are all examples of admissible manifolds.
\medskip

If $(M,g)$ is admissible, points of $M$ can written as $x=(x_1,x')$, where $x_1$ is the Euclidean coordinate. We define
\begin{align*}
\p M_\pm&=\{x\in\p M:\pm\p_\nu\varphi(x)>0\},\\
\p M_{\rm tan}&=\{x\in\p M:\p_\nu\varphi(x)=0\},
\end{align*}
where $\varphi(x)=x_1$. The function $\varphi$ is a natural limiting Carleman weight in $(M,g)$; see \cite{DKSU}. In the results below we assume that there is a part which inaccessible for measurements $\Gamma_{\rm i}\subset \p M_{\rm tan}$, and the accessible part will be denoted by $\Gamma_{\rm a}=\p M_{\rm tan}\setminus \Gamma_{\rm i}$.
\medskip

We say that a unit speed geodesic $\gamma:[0,T]\to M_0$, on transversal simple manifold $(M_0,g_0)$, is \emph{nontangential} if $\gamma(0),\gamma(T)\in\p M_0$ and $\gamma(t)\in M_0^{\rm int}$ if $t\in(0,T)$.
\medskip

The first main result of our paper, says that one can reconstruct the local attenuated geodesic ray transform of the one-dimensional Fourier transform (with respect to $x_1$-variable) of the potential $q$ from the partial knowledge of the Dirichlet-to-Neumann map with $\Gamma_+\supset\p M_+\cup \Gamma_{\rm a}$ and $\Gamma_-\supset\p M_-\cup \Gamma_{\rm a}$.

\begin{Theorem}\label{main th}
Let $(M,g)$ be an admissible manifold, and suppose that $q\in C(M)$ such that $0$ is not a Dirichlet eigenvalue of $-\Delta_g+q$. Let $\Gamma_{\rm i}\subset\p M_{\rm tan}$ be closed such that for some open $E\subset \p M_0$ one has
$$
\Gamma_{\rm i}\subset \R\times(\p M_0\setminus E).
$$
Let $\Gamma_{\rm a}=\p M_{\rm tan}\setminus \Gamma_{\rm i}$ and let $\Gamma_\pm\subset \p M$ be a neighborhood of $\p M_\pm\cup \Gamma_{\rm a}$. Then for any given nontangential geodesic $\gamma:[0,T]\to M_0$ with endpoints on $E$ and for any $\lambda\in\R$, the integral
$$
\int_0^T e^{-2\lambda t}\widehat{(cq)}(2\lambda,\gamma(t))\,dt
$$
can be constructively recovered from the knowledge of $\Lambda_{g,q}(f)$ on $\Gamma_-$ for all $f\in \mathcal H_{g}(\p M)$ supported in $\Gamma_+$. Here $q$ is extended outside of $M$ by zero, and $\widehat{(cq)}$ is the one-dimensional Fourier transform of $q$ with respect to $x_1$-variable.
\end{Theorem}

This is a constructive version of the corresponding uniqueness result by Kenig and Salo \cite[Theorem~2.1]{KSa}.
\medskip

In the next result, we consider the local geodesic ray transform $I_O$ in an open subset $O$ of the transversal simple manifold $(M_0,g_0)$ which is defined for $f\in C(M_0)$ as
$$
I_Of(\gamma):=\int_\gamma f(\gamma(t))\,dt,\quad \gamma\text{ is a nontangential geodesic contained in }O.
$$
We say that $I_O$ is \emph{constructively invertible} in $O$, if any $f\in C(M_0)$ can be recovered in $O$ from the knowledge of $I_Of$.
\medskip

Using Theorem~\ref{main th} one can constructively recover potentials in the set where the local geodesic ray transform is invertible.

\begin{Theorem}\label{main th2}
Suppose that $(M,g)$, $q\in C(M)$, $E\subset\p M_0$ and $\Gamma_\pm$ are as in Theorem~\ref{main th}. Assume that $O\subset M_0$ is open such that $O\cap \p M_0\subset E$ and the local ray transform is constructively invertible on $O$. Then $q$ can be constructively determined in $M\cap(\R\times M_0)$ from the knowledge of $\Lambda_{g,q}(f)$ on $\Gamma_-$ for all $f\in \mathcal H_{g}(\p M)$ supported in $\Gamma_+$.
\end{Theorem}

This result gives a constructive proof of the corresponding uniqueness result by Kenig and Salo \cite[Theorem~2.2]{KSa}; the latter is the above mentioned generalization of the result of Isakov~\cite{Is}.
\medskip

Constructive invertibility of the local ray transform, to the best of author's knowledge, is known in the following case: if $M_0$ has dimension $n\ge 3$ and if $p\in \p M_0$ is such that $\p M_0$ is strictly convex near $p$, then there is an open $O\subset M_0$ containing $p$ on which $I_O$ is constructively invertible; this result is due to Uhlmann and Vasy \cite{UV}. In two dimensions, no such result is known. Even injectivity of the local geodesic ray transform is an open question.
\medskip

If $\p M_{\rm tan}$ has zero measure in $\p M$, we give the reconstruction procedure to determine potentials globally. The problem is reduced to the constructive invertibility of the global geodesic ray transform on the transversal simple manifold $M_0$.
\begin{Theorem}\label{main th3}
Let $(M,g)$ be an admissible manifold, and suppose that $q\in C(M)$ such that $0$ is not a Dirichlet eigenvalue of $-\Delta_g+q$. Suppose that $\p M_{\rm tan}$ is of zero measure in $\p M$. If the global geodesic ray transform is constructively invertible in $M_0$, then $q$ can be constructively determined in $M$ from the knowledge of $\Lambda_{g,q}(f)$ on $\p M_-$ for all $f\in \mathcal H_{g}(\p M)$ supported in $\p M_+$.
\end{Theorem}
This is a generalization with refinements to admissible manifolds of the corresponding result by Nachman and Street \cite{NaS} in Euclidean setting. More precisely, comparing to \cite{NaS}, we do not assume that the subsets of Dirichlet data inputs overlap with the subsets of Neumann data measurements. So our reconstruction procedure is new even in Eulidean space. The version of Theorem~\ref{main th3} was given by Kenig, Salo and Uhlmann \cite{KSaU} for full data case on admissible manifolds of dimension three.
\medskip

Constructive invertibility of the global ray transform is known in the following cases:
\begin{itemize}
\item $(M_0,g_0)=(\overline \Omega,e)$ where $\Omega\subset\R^n$ is open and bounded with $C^\infty$ boundary, and $e$ is the Euclidean metric. In this case inversion formula is given in the book of Sharafutdinov \cite[Section~2.12]{Shar}.

\item $(M_0,g_0)$ of dimension $n\ge 3$, have strictly convex boundary and is globally foliated by strictly convex hypersurfaces. For such case, there is a layer stripping type algorithm for reconstruction developed by Uhlmann and Vasy~\cite{UV}.

\item $(M_0,g_0$ is a simple surface. In this case, there is a Fredholm type inversion formula which was derived by Pestov and Uhlmann \cite{PeU}; see also the article of Krishnan \cite{Kr}.
\end{itemize}
\medskip

The problem of constructive inversion of local or global geodesic ray transforms is of independent interest in integral geometry.
\medskip

The structure of the paper is as follows. In Section~\ref{S2} we give some preliminaries about trace operators and Green's identity for the space $H_{\Delta_g}(M)$. We also consider the well-posedness of the Dirichlet problem for the Schr\"odinger equation $(-\Delta_g+q)u=0$ with boundary condition in $\mathcal H_g(\p M)$. Section~\ref{S4}, following the arguments of \cite{NaS} and modifying them, is devoted to the construction of the new Green's operators for the Laplace-Beltrami operator, and in Section~\ref{S5} the corresponding single layer potentials are constructed. The solvability of the required boundary integral equation is given in Section~\ref{S6}. Then we construct complex geometrical optics solutions in Section~\ref{S7}, and we use these solutions to give reconstruction procedures in Section~\ref{S8}.

\section{Trace operators and the Dirichlet-to-Neumann map}\label{S2}
Let $(M,g)$ be a compact Riemannian manifold with boundary. We use the notation $d\Vol_g$ for the volume form of $(M,g)$ and $d(\p M)_g$ for the induced volume form on the boundary $\p M$. For any two functions $u,v$ on $M$, define an inner product
$$
(u|v)_{L^2(M)}:=\int_M u(x)\overline{v(x)}\,d\Vol_g(x),
$$
and the corresponding norm will be denoted by $\|\cdot\|_{L^2(M)}$. For any two functions $f,h$ on $\Gamma\subset \p M$, define an inner product
$$
(f|h)_\Gamma:=\int_\Gamma f(x)\overline{h(x)}\,d\sigma_{\p M}(x),
$$
and by $\|\cdot\|_\Gamma$ will be denoted the corresponding norm. We also write for short
$$
\|\nabla u\|_{L^2(M)}=\(\int_M |\nabla u(x)|_g^2\,d\Vol_g(x)\)^{1/2}.
$$
Following Bukhgeim and Uhlmann \cite{BU}, we work with the following Hilbert space which is the largest domain of the Laplace-Beltrami operator $\Delta_g$:
$$
H_{\Delta_g}(M)=\{u\in L^2(M):\Delta_g u\in L^2(M)\}.
$$
The norm on $H_{\Delta_g}(M)$ is
$$
\|u\|_{H_{\Delta_g}(M)}^2=\|u\|_{L^2(M)}^2+\|\Delta_g u\|_{L^2(M)}^2.
$$
The proof of the following result is essentially the same as in \cite{BU} (see also, for example \cite{LM}). We include it here for the completeness and accuracy of the exposition.
\begin{Proposition}\label{extensions of the trace operators}
The trace maps $\mathtt{tr}(u)=u|_{\p M}$ and $\mathtt{tr}_\nu (u)=\de{u}{\nu}\big|_{\p M}$ defined on $C^\infty(M)$ have extensions to a bounded operators $H_{\Delta_g}(M)\to H^{-1/2}(\p M)$ and $H_{\Delta_g}(M)\to H^{-3/2}(\p M)$, respectively. Moreover, if $u\in H_{\Delta_g}(M)$ and $\mathtt{tr}(u)\in H^{3/2}(\p M)$, then $u\in H^2(M)$ and $\mathtt{tr}_\nu(u)\in H^{1/2}(\p M)$.
\end{Proposition}
\begin{proof}
First, we show that the trace map $\mathtt{tr}$ has an extension to a bounded operator $H_{\Delta_g}(M)\to H^{-1/2}(\p M)$. Let $u\in C^\infty(M)$ and $w\in H^{1/2}(\p M)$. 
By the surjectivity of the trace map on $H^2(M)$, there is $v\in H^2(M)$ such that
$$
\mathtt{tr}(v)=0,\quad \mathtt{tr}_\nu(v)=\de{v}{\nu}\bigg|_{\p M}=w,\quad \|v\|_{H^2(M)}\le C\|w\|_{H^{1/2}(\p M)}.
$$
Using Green's formula, we get
\begin{align*}
(\mathtt{tr}(u)|w)_{\p M}&=\int_{\p M}\mathtt{tr}(u)\overline{w}\,d(\p M)_g\\
&=\int_{\p M}\mathtt{tr}(u)\overline{\mathtt{tr}_\nu(v)}\,d(\p M)_g=\int_M (u\overline{\Delta_g v}-\overline{v}\Delta_g u)\,d\Vol_g.
\end{align*}
Therefore,
$$
|(\mathtt{tr}(u)|w)_{\p M}|\le \|u\|_{H_{\Delta_g}(M)}\|v\|_{H^2(M)}\le C\|u\|_{H_{\Delta_g}(M)}\|w\|_{H^{1/2}(\p M)}.
$$
This proves that the map $\mathtt{tr}:C^\infty(M)\to H^{-1/2}(\p M)$ is bounded and controlled by the $H_{\Delta_g}(M)$-norm. Since $C^\infty(M)$ is dense in $H_{\Delta_g}(M)$, we can extend $\mathtt{tr}$ to a bounded linear map $H_{\Delta_g}(M)\to H^{-1/2}(\p M)$.
\medskip

Next, we show that the trace map $\mathtt{tr}_\nu$ has an extension to a bounded operator $H_{\Delta_g}(M)\to H^{-3/2}(\p M)$. Let $u\in C^\infty(M)$ and $w\in H^{3/2}(\p M)$. 
By the surjectivity of the trace map on $H^2(M)$, there is $v\in H^2(M)$ such that
$$
\mathtt{tr}(v)=w,\quad \mathtt{tr}_\nu(v)=\de{v}{\nu}\bigg|_{\p M}=0,\quad \|v\|_{H^2(M)}\le C\|w\|_{H^{3/2}(\p M)}.
$$
Using Green's formula, we get
\begin{align*}
(\mathtt{tr}_\nu(u)|w)_{\p M}&=\int_{\p M}\mathtt{tr}_\nu(u)\overline{w}\,d(\p M)_g\\
&=\int_{\p M}\mathtt{tr}_\nu(u)\overline{\mathtt{tr}(v)}\,d(\p M)_g=\int_M (\overline{v}\Delta_g u-u\overline{\Delta_g v})\,d\Vol_g.
\end{align*}
Therefore,
$$
|(\mathtt{tr}_\nu(u)|w)_{\p M}|\le \|u\|_{H_{\Delta_g}(M)}\|v\|_{H^2(M)}\le C\|u\|_{H_{\Delta_g}(M)}\|w\|_{H^{3/2}(\p M)}.
$$
This proves that the map $\mathtt{tr}_\nu:C^\infty(M)\to H^{-3/2}(\p M)$ is bounded and controlled by the $H_{\Delta_g}(M)$-norm. Since $C^\infty(M)$ is dense in $H_{\Delta_g}(M)$, we can extend $\mathtt{tr}_\nu$ to a bounded linear map $H_{\Delta_g}(M)\to H^{-3/2}(\p M)$.
\medskip

Now, we give the proof of the last statement. First, we consider the case when $\mathtt{tr}(u)=0$. Let $u\in C^\infty(M)$ with $\mathtt{tr}(u)=0$. Using, Green's identity, we have
\begin{equation}\label{H^1 norm is bounded by H_Delta norm}
\begin{aligned}
\|u\|_{H^1(M)}^2&=\|u\|_{L^2(M)}^2+\|\nabla u\|_{L^2(M)}^2\\
&=\|u\|_{L^2(M)}^2-(u,\Delta_g u)_{L^2(M)}\\
&\le\|u\|_{L^2(M)}^2+\frac{1}{2}\(\|u\|_{L^2(M)}^2+\|\Delta_g u\|^2_{L^2(M)}\)\\
&\le C\|u\|_{H_{\Delta_g}(M)}^2,
\end{aligned}
\end{equation}
for some constant $C>0$. By \cite[Theorem~1.3]{Tay} in Chapter~5, we have
$$
\|u\|_{H^2(M)}^2\le C\|\Delta_g u\|^2_{L^2(M)}+C\|u\|_{H^1(M)}^2,
$$
for some another constant $C>0$. Combining this with \eqref{H^1 norm is bounded by H_Delta norm}, we obtain
$$
\|u\|_{H^2(M)}^2\le C\|u\|_{H_{\Delta_g}(M)}^2,\quad u\in C^\infty(M),\quad \mathtt{tr}(u)=0.
$$
By density arguments, we obtain
$$
\|u\|_{H^2(M)}^2\le C\|u\|_{H_{\Delta_g}(M)}^2,\quad u\in H_{\Delta_g}(M),\quad \mathtt{tr}(u)=0.
$$
This proves the last statement for the case when $\mathtt{tr}(u)=0$.
\medskip

Suppose now that $u\in H_{\Delta_g}(M)$ with $\mathtt{tr}(u)\in H^{3/2}(\p M)$. By the surjectivity of the trace operator, there is $v\in H^2(M)$ such that $\mathtt{tr}(v)=\mathtt{tr}(u)$. Set $w:=u-v$, then $w\in H_{\Delta_g}(M)$ with $\mathtt{tr}(w)=0$. By what we have proved above, $w\in H^2(M)$, and hence $u\in H^2(M)$.
\end{proof}
The proof of Proposition~\ref{extensions of the trace operators} gives the following.
\begin{Corollary}\label{generalized Green's identity}
For $u\in H_{\Delta_g}(M)$ and $v\in H^2(M)$ we have the generalized Green's identity
\begin{align*}
(u|(-\Delta_g)v)_{L^2(M)}&-((-\Delta_g)u|v)_{L^2(M)}\\
&=\<\mathtt{tr}_\nu(u),\mathtt{tr}(v)\>_{H^{-3/2,3/2}(\p M)}-\<\mathtt{tr}(u),\mathtt{tr}_\nu(v)\>_{H^{-1/2,1/2}(\p M)},
\end{align*}
where $\<\cdot,\cdot\>_{H^{-s,s}(\p M)}$ is the duality between $H^{-s}(\p M)$ and $H^s(\p M)$.
\end{Corollary}

Now we introduce the following space on the boundary $\p M$:
$$
\mathcal H_g(\p M)=\{\mathtt{tr}(u):u\in H_{\Delta_g}(M)\}\subset H^{-1/2}(\p M).
$$
Assume that $q\in L^\infty(M)$ and let us introduce the Bergman space $b_q(M)$ as follows
$$
b_q(M)=\{u\in L^2(M):(-\Delta_g+q)u=0\}\subset H_{\Delta_g}(M).
$$
The topology on this space is a subspace topology in $L^2(M)$. It is not difficult to check that $b_q(M)$ is a closed subspace of $L^2(M)$.
\medskip

We need the following result to define a topology on $\mathcal H_g(\p M)$:
\begin{Proposition}\label{pre solv}
If $q\in L^\infty(M)$ and $0$ is not a Dirichlet eigenvalue of $-\Delta_g+q$ in $M$, then $\mathtt{tr}:b_q(M)\to \mathcal H_g(\p M)$ is one-to-one and onto.
\end{Proposition}
\begin{proof}
Let $u,v\in b_q(M)$ is such that $\mathtt{tr}(u)=\mathtt{tr}(v)$. Set $w=u-v$, then $w\in b_q(M)$ and $\mathtt{tr}(w)=0$. By the last statement of Proposition~\ref{extensions of the trace operators}, $w\in H^2(M)$. By assumption, $0$ is not a Dirichlet eigenvalue of $-\Delta_g+q$ in $M$. Therefore, $(-\Delta_g+q)w=0$ with $w|_{\p M}=0$ imply that $w=0$.
\medskip

Let $h\in\mathcal H_g(\p M)$. By definition of $\mathcal H_g(\p M)$, there is $u\in H_{\Delta_g}(M)$ such that $\mathtt{tr}(u)=h$. Take $v\in H^1_0(M)$ being the solution to the Dirichlet problem $(-\Delta_g+q)v=(-\Delta_g+q)u$, $v|_{\p M}=0$. Set $w=u-v$. Then $w\in H_{\Delta_g}(M)$, $(-\Delta_g+q)w=0$ and $\mathtt{tr}(w)=0$. In other words, $w\in b_q(M)$ with $\mathtt{tr}(w)=h$.
\end{proof}
Let $P_q$ denote the inverse of $\mathtt{tr}:b_q(M)\to \mathcal H_g(\p M)$. We define the norm on $\mathcal H_g(\p M)$ as
$$
\|f\|_{\mathcal H_g(\p M)}=\|P_0 f\|_{L^2(M)}.
$$
In particular, by Proposition~\ref{pre solv}, this implies that $\mathtt{tr}:b_0\to\mathcal H_g(\p M)$ as well as $P_0:\mathcal H_g(\p M)\to b_0$ are bounded. Next, we give the following solvability result of the Dirichlet problem with boundary data in $\mathcal H_g(\p M)$:
\begin{Proposition}\label{solvability of the Dirichlet problem}
The operator $\mathtt{tr}:H_{\Delta_g}(M)\to \mathcal H_g(\p M)$ is bounded. If $q\in L^\infty(M)$ and $0$ is not a Dirichlet eigenvalue of $-\Delta_g+q$ in $M$, then $\mathtt{tr}:b_q(M)\to \mathcal H_g(\p M)$ is a homeomorphism.
\end{Proposition}
\begin{proof}
Let $u\in H_{\Delta_g}(M)$. Consider $v\in H^1_0(M)$ being the solution to the Dirichlet problem $(-\Delta_g)v=(-\Delta_g)u$, $v|_{\p M}=0$. Set $w=u-v$. Note that $u\mapsto v$ is bounded $H_{\Delta_g}(M)\to L^2(M)$, and hence $u\mapsto w$ is bounded $H_{\Delta_g}(M)\to b_0(M)$ as well, since $(-\Delta_g)w=0$. Since $\mathtt{tr}:b_0\to\mathcal H_g(\p M)$ is bounded and since $\mathtt{tr}(w)=\mathtt{tr}(u)$, we can conclude that the map $u\mapsto \mathtt{tr}(u)$ is bounded $H_{\Delta_g}(M)\to \mathcal H_g(\p M)$.
\medskip

Since the inclusion $b_q\hookrightarrow H_{\Delta_g}(M)$ is bounded, by the first part of the proposition, the map $\mathtt{tr}:b_q\to \mathcal H_g(M)$ is bounded. Bijectivity of the latter map, which follows from Proposition~\ref{pre solv}, together with Open Mapping Theorem, implies the last statement.
\end{proof}

We also extend the domain of the Dirichlet-to-Neumann map to $\mathcal H_g(\p M)$:
\begin{Proposition}\label{extending difference of DN maps}
Suppose that $q\in L^\infty(M)$ and $0$ is not a Dirichlet eigenvalue of $-\Delta_g+q$ in $M$. Then $(\Lambda_q-\Lambda_0)|_{\mathcal H(\p M)}$ is a bounded operator $\mathcal H_g(\p M)\to(\mathcal H_g(\p M))^*$. Moreover, the following integral identity holds
\begin{equation}\label{main integral identity}
\<h,(\Lambda_{g,q}-\Lambda_{g,0})f\>_{H^{-1/2,1/2}(\p M)}=(P_0(h)|qP_q(f))_{L^2(M)},
\end{equation}
for all $f,h\in \mathcal H_g(\p M)$.
\end{Proposition}
\begin{proof}
Suppose that $f,h\in \mathcal H_g(\p M)$. Let $u\in H_{\Delta_g}(M)$ be the unique solution to the boundary value problem
$$
(-\Delta_g+q)u=0\,\text{ \rm in }\Omega,\quad \mathtt{tr}(u)=f,
$$
and let $u_0$ be the unique solution to the boundary value problem
$$
(-\Delta_g)u_0=0\,\text{ \rm in }\Omega,\quad \mathtt{tr}(u_0)=f.
$$
Set $w:=u-u_0$, then we have
$$
(-\Delta_g)w=-qu\,\text{ \rm in }\Omega,\quad \mathtt{tr}(w)=0.
$$
By the last statement of Proposition~\ref{extensions of the trace operators}, we can conclude that $w\in H^2(M)$. Note that by Proposition~\ref{solvability of the Dirichlet problem}, there is $v_h\in H_{\Delta_g}(M)$ such that $(-\Delta_g)v_h=0$ and $\mathtt{tr}(v_h)=h$. Now, we can apply Corollary~\ref{generalized Green's identity} and get
\begin{align*}
(v_h|qu)_{L^2(M)}&=-(v_h|(-\Delta_g)w)_{L^2(M)}\\
&=-((-\Delta_g)v_h|w)_{L^2(M)}+\<h,\mathtt{tr}_\nu(w)\>_{H^{-1/2,1/2}(\p M)}.
\end{align*}
Since $(-\Delta_g)v_h=0$ and $\mathtt{tr}_\nu(w)=\mathtt{tr}_\nu(u-u_0)=(\Lambda_{g,q}-\Lambda_{g,0})f$, we obtain
\begin{equation}\label{main integral identity'}
\<h,(\Lambda_{g,q}-\Lambda_{g,0})f\>_{H^{-1/2,1/2}(\p M)}=(v_h|qu)_{L^2(M)}.
\end{equation}
The right-hand side depends continuously on $f,h\in \mathcal H_g(\p M)$. Hence, so does the left hand-side and this together with~\eqref{main integral identity'} implies that the result.
\end{proof}

\section{The Green's operators}\label{S4}
Let $(M,g)$ be an admissible manifold and let $q\in L^\infty(M)$. Let us introduce certain notations which will be used thoughout the paper. For $\tau\in\R$, we consider the following disjoint decomposition $\p M=S^+_\tau\cup S^-_\tau$, where
$$
S^+_\tau:=\{x\in \p M:\sgn(\tau)\p_\nu \varphi(x)\ge|3\tau|^{-1}\},\quad S^-_\tau:=\p M\setminus S^+_\tau.
$$
For $\delta>0$,  we can write $S^-_\tau=S^-_{\tau,\delta}\cup S^0_{\tau,\delta}$, where
\begin{align*}
S^-_{\tau,\delta}&:=\{x\in \p M:\sgn(\tau)\p_\nu \varphi(x)\le-\delta\},\\
S^0_{\tau,\delta}&:=\{x\in \p M:-\delta<\sgn(\tau)\p_\nu \varphi(x)<(3|\tau|)^{-1}\}.
\end{align*}

Constructions of Green's operators and the corresponding single layer potentials, as well as construction of complex geometrical optics solutions are based on the following Carleman estimates with boundary terms for the conjugated operator
$$
e^{\tau x_1}(-\Delta_g+q)e^{-\tau x_1}.
$$
\begin{Proposition}
Let $(M,g)$ be an admissible manifold and let $q\in L^\infty(M)$. There are constants $C_0,\tau_0>$ such that for all $\tau\in\R$ with $|\tau|\ge \tau_0$ and $\delta>0$, we have
\begin{multline}\label{Carleman estimate}
(\delta|\tau|)^{1/2}\|\p_\nu u\|_{S^-_{\tau,\delta}}+\|\p_\nu u\|_{S^0_{\tau,\delta}}+|\tau|\|u\|_{L^2(M)}+\|\nabla u\|_{L^2(M)}\\
\le C_0 \|e^{\tau x_1}(-\Delta_g+q)e^{-\tau x_1}u\|_{L^2(M)}+C_0|\tau|^{1/2}\|\p_\nu u\|_{S^+_\tau}
\end{multline}
for all $u\in C^\infty(M)$ with $u|_{\p M}=0$.
\end{Proposition}
\begin{proof}
This estimate was proven by Kenig and Salo; see \cite[Proposition~4.2]{KSa}.
\end{proof}
Define
$$
\mathcal D^\pm_\tau=\{u\in C^\infty(M):u|_{\p M}=\mathtt{tr}_\nu(u)|_{S^\pm_\tau}=0\}.
$$

The aim of this section is to prove the following result.
\begin{Theorem}\label{existence of G_tau operator}
Let $(M,g)$ be an admissible manifold. There is a constant $\tau_0>0$ such that for all $\tau\in\R$ with $|\tau|\ge\tau_0$, there is a linear operator
$$
G_\tau:L^2(M)\to L^2(M)
$$
such that
$$
e^{\tau x_1}(-\Delta_g)e^{-\tau x_1}G_\tau v=v,\quad v\in L^2(M)
$$
and
\begin{equation}\label{main difficulty in the proof of existence of G_tau}
G_\tau e^{\tau x_1}(-\Delta_g)e^{-\tau x_1}u=u,\quad u\in\mathcal D^+_{-\tau}.
\end{equation}
This operator satisfies
$$
\|G_\tau f\|_{L^2(M)}\le \frac{C_0}{|\tau|}\|f\|_{L^2(M)},\quad f\in L^2(M),
$$
where $C_0>0$ is independent of $\tau$. Moreover, $G_\tau:L^2(M)\to e^{\tau x_1}H_{\Delta_g}(M)$ and for all $v\in L^2(M)$ support of $\mathtt{tr}(G_\tau v)$ is in $S^+_{\tau}$.
\end{Theorem}

Let $\pi_\tau$ be the orthogonal projection onto $\mathcal L_{\tau}$ the closure of $e^{-\tau x_1}(-\Delta_g)e^{\tau x_1}\mathcal D^+_{\tau}$ in $L^2(M)$.

\begin{Lemma}\label{orthogonal projection of pi^perp_tau}
Let $\pi_\tau^\perp:=\id-\pi_\tau$. Then $\pi_\tau^\perp$ is the orthogonal projection onto
$$
\mathcal A_\tau=\{u\in L^2(M):e^{\tau x_1}(-\Delta_g)e^{-\tau x_1} u=0\text{ \rm and }\mathtt{tr}(u)\text{ \rm is supported in }S^+_{\tau}\}.
$$
\end{Lemma}
\begin{proof}
It is enough to show that $u$ is orthogonal to $e^{-\tau x_1}(-\Delta_g)e^{\tau x_1}\mathcal D^+_{\tau}$ if and only if $u$ is in $\mathcal A_\tau$. Suppose that $u\in \mathcal A_\tau$. Then for $v\in \mathcal D^+_{\tau}$, we have
$$
(u|e^{-\tau x_1}(-\Delta_g)e^{\tau x_1}v)_{L^2(M)}=(e^{\tau x_1}(-\Delta_g)e^{-\tau x_1}u|v)_{L^2(M)}-(\mathtt{tr}(u),\mathtt{tr}_\nu(v))_{H^{-1/2,1/2}(\p M)}.
$$
Since $\mathtt{tr}(u)$ supported in $S^+_{\tau}$, $\mathtt{tr}_\nu(v)=0$ in $S^+_{\tau}$ and $e^{\tau x_1}(-\Delta_g)e^{-\tau x_1} u=0$, we obtain 
$$
(u|e^{-\tau x_1}(-\Delta_g)e^{\tau x_1}v)_{L^2(M)}=0,
$$
which means that $u$ is orthogonal to $e^{-\tau x_1}(-\Delta_g)e^{\tau x_1}\mathcal D^+_{\tau}$. Converse is as in~\cite[Lemma~3.3]{NaS}.
\end{proof}

\begin{Proposition}\label{solvability result giving H_tau}
Let $(M,g)$ be an admissible manifold. There is $\tau_0>0$ such that for all $\tau\in\R$ with $|\tau|\ge \tau_0$ and for a given $v\in L^2(M)$, there is a unique solution $u\in L^2(M)$ of the equation
$$
e^{\tau x_1}(-\Delta_g)e^{-\tau x_1} u=v\quad\text{in}\quad M
$$
such that $\mathtt{tr}(u)$ is supported in $S^+_{\tau}$, $\pi_\tau u=u$ and $\|u\|_{L^2(M)}\le C_0 \frac{1}{|\tau|}\|v\|_{L^2(M)}$ with constant $C_0>0$ independent of $\tau$.
\end{Proposition}
\begin{proof}
First, we show the existence. Define a linear functional $L$ on $e^{-\tau x_1}(-\Delta_g)e^{\tau x_1}\mathcal D^+_{\tau}$ by
$$
L(e^{-\tau x_1}(-\Delta_g)e^{\tau x_1} w)=(v|w)_{L^2(M)},\quad w\in \mathcal D^+_{\tau}.
$$
Then we have
\begin{align*}
|L(e^{-\tau x_1}(-\Delta_g)e^{\tau x_1} w)|&\le \|v\|_{L^2(M)}\|w\|_{L^2(M)}\\
&\le C_0 \frac{1}{|\tau|}\|v\|_{L^2(M)}\|e^{\tau x_1}(-\Delta_g)e^{-\tau x_1}w\|_{L^2(M)},
\end{align*}
where in the last step we have used the Carleman estimate \eqref{Carleman estimate}. By the Hahn-Banach theorem, we may extend $L$ to a linear continuous functional $\widetilde L$ on $\mathcal L_{\tau}$. On the orthogonal complement of $\mathcal L_{\tau}$ in $L^2(M)$ we define $\widetilde L$ to be zero. By the Riesz representation theorem, there exists $u\in L^2(M)$ such that
$$
\widetilde L(f)=(u|f)_{L^2(M)},\quad f\in L^2(M).
$$
In particular,
\begin{equation}\label{widetilde L acting on D_sgn}
\begin{aligned}
(u|e^{-\tau x_1}(-\Delta_g)e^{\tau x_1} w)_{L^2(M)}&=\widetilde L(e^{-\tau x_1}(-\Delta_g)e^{\tau x_1}w)\\
&=(v|w)_{L^2(M)},\qquad w\in\mathcal D^+_{\tau}.
\end{aligned}
\end{equation}
If we take $w\in C^\infty_0(M^{\rm int})$ in the above equation, we obtain $e^{-\tau x_1}(-\Delta_g)e^{\tau x_1}u=v$. Moreover,
$$
\|u\|_{L^2(M)}\le C_0 \frac{1}{|\tau|}\|v\|_{L^2(M)}.
$$
Since $\widetilde L\equiv 0$ on the orthogonal complement of $\mathcal L_{\tau}$ in $L^2(M)$, we have that $u\in\mathcal L_{\tau}$ and hence $\pi_\tau u=u$.
\medskip

To finish the proof, we need to show that $\mathtt{tr}(u)$ is supported in $S^+_{\tau}$. For arbitrary $w\in \mathcal D^+_{\tau}$, using the generalized Green's identity from Corollary~\ref{generalized Green's identity}, we get
$$
(u|e^{-\tau x_1}(-\Delta_g)e^{\tau x_1} w)_{L^2(M)}=\<\mathtt{tr}(u),\mathtt{tr}_\nu(w)\>_{H^{-1/2,1/2}(\p M)}+(v|w)_{L^2(M)}.
$$
According to \eqref{widetilde L acting on D_sgn}, we have $\<\mathtt{tr}_\nu(w),\mathtt{tr}(u)\>_{\p M}=0$. Since $w\in \mathcal D^+_{\tau}$ was arbitrary, we can conclude that $\mathtt{tr}(u)$ is supported in $S^+_{\tau}$.
\medskip

Now, we prove uniqueness. Suppose that $u'\in L^2(M)$ is another solution of the equation $e^{\tau x_1}(-\Delta_g)e^{-\tau x_1} u'=v$ satisfying all the conditions of the proposition. Then $e^{\tau x_1}(-\Delta_g)e^{-\tau x_1}(u-u')=0$, $\mathtt{tr}(u-u')$ is supported in $S^+_{\tau}$, and $\pi_\tau(u-u')=u-u'$. However, by Lemma~\ref{orthogonal projection of pi^perp_tau}, $\pi_\tau(u-u')=0$. Thus, we obtain $u-u'=0$ which finishes the proof.
\end{proof}

Let $H_\tau:L^2(M)\to L^2(M)$ be the solution operator obtained in the previous result. In other words, the operator $H_\tau$ is defined by $H_\tau v=u$, where $u$ and $v$ are as in Proposition~\ref{solvability result giving H_tau}. The following is an immediate corollary of the preceeding result.

\begin{Corollary}\label{corollary of solvability result giving H_tau}
Let $(M,g)$ be an admissible manifold. There is $\tau_0>0$ such that for all $\tau\in\R$ with $|\tau|\ge \tau_0$, there is a linear operator
$$
H_\tau:L^2(M)\to L^2(M)
$$
such that
$$
e^{\tau x_1}(-\Delta_g)e^{-\tau x_1}H_\tau v=v,\quad v\in L^2(M)
$$
and $\pi_\tau H_\tau=H_\tau$. This operator satisfies
$$
\|H_\tau f\|_{L^2(M)}\le C_0 \frac{1}{|\tau|}\|f\|_{L^2(M)}
$$
where $C_0>0$ is independent of $\tau$. Moreover, $H_\tau:L^2(M)\to e^{\tau x_1}H_{\Delta_g}(M)$ and for all $v\in L^2(M)$ support of $\mathtt{tr}(H_\tau v)$ is in $S^+_{\tau}$.
\end{Corollary}

Thus, the operator $H_\tau$ satisfies Theorem~\ref{existence of G_tau operator} except \eqref{main difficulty in the proof of existence of G_tau}. We shall accodingly modify $H_\tau$ to obtain \eqref{main difficulty in the proof of existence of G_tau}. We need the technical result.
\begin{Lemma}\label{technical lemma}
Let $T_\tau:=H_\tau \pi_{-\tau}$. Then $T_\tau^*=T_{-\tau}$.
\end{Lemma}
\begin{proof}
Note that $T_\tau^*\pi^\perp_\tau=\pi_{-\tau} H_\tau^*(\id-\pi_\tau)=\pi_{-\tau} H_\tau^*-\pi_{-\tau} H_\tau^*\pi_\tau=0$, where in the first step we have used the fact that $\pi_{-\tau}^*=\pi_{-\tau}$ (since $\pi_{-\tau}$ is projection) and in the last step we have used that $H_\tau^*\pi_\tau=H_\tau^*$ (this follows from $\pi_\tau H_\tau=H_\tau$ which is true by Corollary~\ref{corollary of solvability result giving H_tau}). Also note that $T_{-\tau}\pi_\tau^\perp=H_{-\tau}\pi_\tau\pi_\tau^\perp=0$. Thus, $T_\tau^*\pi^\perp_\tau=T_{-\tau}\pi_\tau^\perp=0$, and hence, to prove the lemma it is sufficient to show that
$$
T^*_\tau e^{-\tau x_1}(-\Delta_g)e^{\tau x_1} v=H_{-\tau}e^{-\tau x_1}(-\Delta_g)e^{\tau x_1} v
$$
for all $\mathcal D^+_{\tau}$. Observe that $\pi_{-\tau}T_\tau^*=\pi_{-\tau}^2 H_\tau^*=\pi_{-\tau} H_\tau^*=T^*_{-\tau}$ (since $\pi_{-\tau}^2=\pi_{-\tau}$). Therefore, it is enough to show that
\begin{multline*}
(e^{\tau x_1}(-\Delta_g)e^{-\tau x_1}w|T^*_\tau e^{-\tau x_1}(-\Delta_g)e^{\tau x_1}v)_{L^2(M)}\\
=(e^{\tau x_1}(-\Delta_g)e^{-\tau x_1}w|H_{-\tau} e^{-\tau x_1}(-\Delta_g)e^{\tau x_1}v)_{L^2(M)},
\end{multline*}
for all $w\in \mathcal D^+_{-\tau}$ and for all $v\in \mathcal D^+_{\tau}$. We have
\begin{align*}
(e^{\tau x_1}(-\Delta_g)e^{-\tau x_1}w|&T^*_\tau e^{-\tau x_1}(-\Delta_g)e^{\tau x_1}v)_{L^2(M)}\\
&=(e^{\tau x_1}(-\Delta_g)e^{-\tau x_1}w|\pi_{-\tau}H^*_\tau e^{-\tau x_1}(-\Delta_g)e^{\tau x_1}v)_{L^2(M)}\\
&=(H_\tau \pi_{-\tau}e^{\tau x_1}(-\Delta_g)e^{-\tau x_1}w| e^{-\tau x_1}(-\Delta_g)e^{\tau x_1}v)_{L^2(M)}\\
&=(H_\tau e^{\tau x_1}(-\Delta_g)e^{-\tau x_1}w| e^{-\tau x_1}(-\Delta_g)e^{\tau x_1}v)_{L^2(M)}.
\end{align*}
Since by Corollary~\ref{corollary of solvability result giving H_tau} we know that $\mathtt{tr}(H_\tau e^{\tau x_1}(-\Delta_g)e^{-\tau x_1}w)$ is supported in $S^+_{\tau}$, we can use Green's identity and the fact that $v\in\mathcal D^+_{\tau}$ to get
\begin{align*}
(e^{\tau x_1}(-\Delta_g)e^{-\tau x_1}w|&T^*_\tau e^{-\tau x_1}(-\Delta_g)e^{\tau x_1}v)_{L^2(M)}\\
&=(e^{\tau x_1}(-\Delta_g)e^{-\tau x_1}H_\tau e^{\tau x_1}(-\Delta_g)e^{-\tau x_1}w|v)_{L^2(M)}
\end{align*}
Since $e^{\tau x_1}(-\Delta_g)e^{-\tau x_1}H_\tau=\id$ by Corollary~\ref{corollary of solvability result giving H_tau}, we obtain
\begin{align*}
(e^{\tau x_1}(-\Delta_g)e^{-\tau x_1}w|&T^*_\tau e^{-\tau x_1}(-\Delta_g)e^{\tau x_1}v)_{L^2(M)}\\
&=(e^{\tau x_1}(-\Delta_g)e^{-\tau x_1}w|v)_{L^2(M)}\\
&=(w|e^{-\tau x_1}(-\Delta_g)e^{\tau x_1}v)_{L^2(M)}.
\end{align*}
Here, in the last step we used the Green's identity and that $w|_{\p M}=v|_{\p M}=0$.
\medskip

Using that $\mathtt{tr}(H_{-\tau} e^{-\tau x_1}(-\Delta_g)e^{\tau x_1}v)$ is supported in $S^+_{-\tau}$, $w\in\mathcal D_{-\tau}^+$ and the Green's identity, we obtain
\begin{align*}
(e^{\tau x_1}(-\Delta_g)e^{-\tau x_1}w|H_{-\tau} e^{-\tau x_1}&(-\Delta_g)e^{\tau x_1}v)_{L^2(M)}\\
&=(w|e^{-\tau x_1}(-\Delta_g)e^{\tau x_1}H_{-\tau} e^{-\tau x_1}(-\Delta_g)e^{\tau x_1}v)_{L^2(M)}\\
&=(w|e^{-\tau x_1}(-\Delta_g)e^{\tau x_1}v)_{L^2(M)}.
\end{align*}
In the last step we used the fact that $e^{-\tau x_1}(-\Delta_g)e^{\tau x_1}H_{-\tau}=\id$ (by Corollary~\ref{corollary of solvability result giving H_tau}). The proof of the lemma is thus complete.
\end{proof}

\begin{proof}[Proof of Theorem~\ref{existence of G_tau operator}]
Define $G_\tau=H_\tau+\pi_\tau^\perp H_{-\tau}^*$. By Corollary~\ref{corollary of solvability result giving H_tau} and Lemma~\ref{orthogonal projection of pi^perp_tau}, it follows that
$$
e^{\tau x_1}(-\Delta_g)e^{-\tau x_1} G_\tau v=v,\quad v\in L^2(M),
$$
$$
\|G_\tau f\|_{L^2(M)}\le C_0\frac{1}{|\tau|}\|f\|_{L^2(M)},\quad f\in L^2(M),
$$
$G_\tau:L^2(M)\to e^{\tau x_1}H_{\Delta_g}(M)$ and that for all $v\in L^2(M)$ support of $\mathtt{tr}(G_\tau v)$ is in~$S^+_{\tau}$.
\medskip

It is left to prove \eqref{main difficulty in the proof of existence of G_tau}. For this, we need first to show that $G_\tau^*=G_{-\tau}$. Using Lemma~\ref{technical lemma}, we can show
$$
G_\tau^*=H^*_\tau+H_{-\tau}\pi_\tau^\perp=(H_\tau\pi_{-\tau}^\perp+T_\tau)^*+H_{-\tau}-T_{-\tau}=H_{-\tau}+\pi_{-\tau}^\perp H^*_\tau=G_{-\tau}.
$$
Using this, for $f\in L^2(M)$ and $u\in \mathcal D^+_{-\tau}$, we have
$$
(f|G_{\tau}e^{\tau x_1}(-\Delta_g)e^{-\tau x_1}u)_{L^2(M)}=(G_{-\tau}f|e^{\tau x_1}(-\Delta_g)e^{-\tau x_1}u)_{L^2(M)}.
$$
We have shown that $\mathtt{tr}(G_{-\tau}f)$ is supported in $S^+_{-\tau}$. This fact together with $u\in\mathcal D^+_{-\tau}$ allows us to use the generalized Green's identity from Corollary~\ref{generalized Green's identity} and get
$$
(f|G_{\tau}e^{\tau x_1}(-\Delta_g)e^{-\tau x_1}u)_{L^2(M)}=(e^{-\tau x_1}(-\Delta_g)e^{\tau x_1}G_{-\tau}f|u)_{L^2(M)}=(f|u)_{L^2(M)}.
$$
Here, in the last step we used the already proven fact that $e^{-\tau x_1}(-\Delta_g)e^{\tau x_1} G_{-\tau}=\id$. This finishes the proof.
\end{proof}

\section{Single layer operators}\label{S5}
The aim of this section is to construct the single layer operators $S_\tau$ corresponding to the Green's operators $G_\tau$ constructed in the previous section.
\medskip

Let $\tau_0>0$ be as in Theorem~\ref{existence of G_tau operator}. For $\tau\in\R$ with $|\tau|\ge \tau_0$, consider the operator
$$
(\mathtt{tr}\circ G_\tau)^*:e^{-\tau x_1}(\mathcal H_g(\p M))^*\to L^2(M).
$$
In other words, $(\mathtt{tr}\circ G_\tau)^*$ defined for $h\in e^{-\tau x_1}(\mathcal H_g(\p M))^*$ by
$$
(f|(\mathtt{tr}\circ G_\tau)^*h)_{L^2(M)}=\<(\mathtt{tr}\circ G_\tau)f,h\>_{H^{-1/2,1/2}(\p M)},\quad f\in L^2(M).
$$
\begin{Proposition}\label{properties of tr G_tau map}
For all $h\in e^{-\tau x_1}(\mathcal H_g(\p M))^*$ we have
$$
e^{-\tau x_1}(-\Delta_g)e^{\tau x_1} (\mathtt{tr}\circ G_\tau)^*h=0
$$
and the support of $\mathtt{tr}((\mathtt{tr}\circ G_\tau)^*h)$ is in $S^+_{-\tau}$. Moreover, suppose that $B$ is a neighborhood of $S^+_{\tau}$ such that $\overline B\subset \p M_{\sgn(\tau)}$, and that the support of $h$ is in $\p M\setminus B$. Then $(\mathtt{tr}\circ G_\tau)^*h=0$.
\end{Proposition}
\begin{proof}
Let $h\in e^{-\tau x_1}(\mathcal H_g(\p M))^*$. Then for all $f\in\mathcal D^+_{-\tau}$, we obtain
\begin{equation}\label{eq:6.1}
\begin{aligned}
(e^{\tau x_1}(-\Delta_g)e^{-\tau x_1}f|(\mathtt{tr}\circ G_\tau&)^*h)_{L^2(M)}\\
&=\<(\mathtt{tr}\circ G_\tau)e^{\tau x_1}(-\Delta_g)e^{-\tau x_1}f,h\>_{H^{-1/2,1/2}(\p M)}\\
&=\<\mathtt{tr}(f),h\>_{H^{-1/2,1/2}(\p M)}=0.
\end{aligned}
\end{equation}
Here, we have used \eqref{main difficulty in the proof of existence of G_tau} and $\mathtt{tr}(f)=0$. If we take $f\in C^\infty_0(M)$ in \eqref{eq:6.1} and use the generalized Green's identity in Corollary~\ref{generalized Green's identity}, we can show that
$$
(f|e^{-\tau x_1}(-\Delta_g)e^{\tau x_1} (\mathtt{tr}\circ G_\tau)^*h)_{L^2(M)}=(e^{\tau x_1}(-\Delta_g)e^{-\tau x_1}f|(\mathtt{tr}\circ G_\tau)^*h)_{L^2(M)}=0.
$$
Hence, we obtain $e^{-\tau x_1}(-\Delta_g)e^{\tau x_1} (\mathtt{tr}\circ G_\tau)^*h=0$ for all $h\in e^{-\tau x_1}(\mathcal H_g(\p M))^*$.
\medskip

Let us now show that the support of $\mathtt{tr}((\mathtt{tr}\circ G_\tau)^*h)$ is in $S^+_{-\tau}$. For arbitrary $f\in \mathcal D^+_{-\tau}$, using the generalized Green's identity from Corollary~\ref{generalized Green's identity}, we get
\begin{align*}
\<\mathtt{tr}((\mathtt{tr}\circ G_\tau)^*h),\mathtt{tr}_\nu(f)&\>_{H^{-1/2,1/2}(\p M)}\\
&=\<\mathtt{tr}((\mathtt{tr}\circ G_\tau)^*h),\mathtt{tr}_\nu(f)\>_{H^{-1/2,1/2}(\p M)}\\
&\quad-\<\mathtt{tr}_\nu((\mathtt{tr}\circ G_\tau)^*h),\mathtt{tr}(f)\>_{H^{-3/2,3/2}(\p M)}\\
&=(e^{-\tau x_1}(-\Delta_g) e^{\tau x_1}(\mathtt{tr}\circ G_\tau)^*h|f)_{L^2(M)}\\
&\quad-((\mathtt{tr}\circ G_\tau)^*h|e^{\tau x_1}(-\Delta_g) e^{-\tau x_1}f)_{L^2(M)}\\
&=-((\mathtt{tr}\circ G_\tau)^*h|e^{\tau x_1}(-\Delta_g) e^{-\tau x_1}f)_{L^2(M)}=0,
\end{align*}
where in the last step we have used \eqref{eq:6.1}.
\medskip

Now, we prove the last statement of the proposition. If $h$ is supported in $\p M\setminus B$, then for all $f\in L^2(M)$ we have
$$
(f|(\mathtt{tr}\circ G_\tau)^*h)_{L^2(M)}=\<(\mathtt{tr}\circ G_\tau)f,h\>_{H^{-1/2,1/2}(S^+_{\tau})}=0.
$$
This is because by the last statement of Theorem~\ref{existence of G_tau operator}, $(\mathtt{tr}\circ G_\tau)f$ is supported in $S^+_{\tau}$. The proof of the proposition is thus complete.
\end{proof}

For $\tau\in\R$ with $|\tau|\ge\tau_0$, define the operator $S_\tau$ for $h\in (\mathcal H_g(\p M))^*$ by
$$
S_\tau h=e^{-\tau x_1}(\mathtt{tr}\circ (\mathtt{tr}\circ G_\tau)^*)^*(e^{\tau x_1}h).
$$
\begin{Proposition}\label{properties of S_tau operator}
For $\tau\in\R$ with $|\tau|\ge\tau_0$, the operator $S_\tau$ is bounded $(\mathcal H_g(\p M))^*\to \mathcal H_g(\p M)$, and for $h\in(\mathcal H(\p M))^*$, $S_\tau h$ depends only on $h|_{\p M_{-\sgn(\tau)}}$ and supported in~$B$.
\end{Proposition}
\begin{proof}
By Proposition~\ref{properties of tr G_tau map}, we have that the operator $(\mathtt{tr}\circ G_\tau)^*:e^{-\tau x_1}(\mathcal H_g(\p M))^*\to e^{-\tau x_1}H_{\Delta_g}(M)$ is bounded, and hence the operator $\mathtt{tr}\circ (\mathtt{tr}\circ G_\tau)^*:e^{-\tau x_1}(\mathcal H_g(\p M))^*\to e^{-\tau x_1}\mathcal H_g(\p M)$ is bounded as well. This implies the boundedness of $(\mathtt{tr}\circ (\mathtt{tr}\circ G_\tau)^*)^*:e^{\tau x_1}(\mathcal H_g(\p M))^*\to e^{\tau x_1}\mathcal H_g(\p M)$. Therefore, $S_\tau$ is a bounded operator $(\mathcal H_g(\p M))^*\to \mathcal H_g(\p M)$.
\medskip

We have by Proposition~\ref{properties of tr G_tau map} that $\mathtt{tr}\circ (\mathtt{tr}\circ G_\tau)^*(e^{-\tau x_1}\tilde h)$ is supported in $S^+_{-\tau}$ if $\tilde h\in (\mathcal H_g(\p M))^*$. By duality, for $h\in (\mathcal H_g(\p M))^*$, $S_\tau h=e^{-\tau x_1}(\mathtt{tr}\circ (\mathtt{tr}\circ G_\tau)^*)^*(e^{\tau x_1}h)$ depends only on $h|_{\p M_{-\sgn(\tau)}}$.
\medskip

By the last statement of Proposition~\ref{properties of tr G_tau map}, if $\tilde h\in (\mathcal H_g(\p M))^*$ is supported in $\p M\setminus B$ then $\mathtt{tr}((\mathtt{tr}\circ G_\tau)^*(e^{-\tau x_1}\tilde h))=0$. By duality, for any $h\in (\mathcal H_g(\p M))^*$, $e^{-\tau x_1}(\mathtt{tr}\circ (\mathtt{tr}\circ G_\tau)^*)^*(e^{\tau x_1}h)$ supported in $B$.
\end{proof}

\section{Boundary integral equation}\label{S6}
In the present section, we prove the solvability of the following boundary integral equation: for $\tau\in\R$ with $|\tau|\ge\tau_0$
\begin{equation}\label{boundary integral equation}
\(\id+S_\tau(\Lambda_{g,q}-\Lambda_{g,0})\)h=f,\quad f,h\in \mathcal H_g(\p M).
\end{equation}
To prove the solvability of \eqref{boundary integral equation}, we need the following result on basic properties of the operator $S_\tau(\Lambda_{g,q}-\Lambda_{g,0})$.
\begin{Proposition}\label{properties of S_tau(Lambda_q-Lambda_0)}
Suppose that $q\in L^\infty(M)$ and $0$ is not a Dirichlet eigenvalue of $-\Delta_g+q$ in $M$. There is $\tau_0>0$ such that for all $\tau\in\R$ with $|\tau|\ge\tau_0$, the operator $S_\tau(\Lambda_{g,q}-\Lambda_{g,0})$ is a bounded operator $\mathcal H_g(\p M)\to \mathcal H_g(\p M)$, and for $f\in\mathcal H_g(\p M)$, $S_\tau(\Lambda_{g,q}-\Lambda_{g,0})f$ is supported in $B$ and can be computed from the knowledge of $\Lambda_q f|_{\p M_{-\sgn(\tau)}}$. Moreover, the following factorization identity holds
$$
S_\tau (\Lambda_{g,q}-\Lambda_{g,0})=\mathtt{tr}\circ e^{-\tau x_1}G_\tau e^{\tau x_1}q P_q.
$$
\end{Proposition}
\begin{proof}
First part of the proposition is a consequence of Proposition~\ref{extending difference of DN maps} and Proposition~\ref{properties of S_tau operator}. To prove the last statement, consider $h\in (\mathcal H_g(\p M))^*$ and $f\in \mathcal H_g(\p M)$. Then
$$
\<h,\mathtt{tr}\circ e^{-\tau x_1}G_\tau e^{\tau x_1}q P_q(f)\>_{H^{-1/2,1/2}(\p M)}=(e^{\tau x_1}(\mathtt{tr}\circ G_\tau)^*(e^{-\tau x_1}h)|qP_q(f))_{L^2(M)}.
$$
By Proposition~\ref{properties of tr G_tau map}, $e^{\tau x_1}(\mathtt{tr}\circ G_\tau)^*(e^{-\tau x_1}h)$ is in $b_0$. Using \eqref{main integral identity}, we show that
\begin{align*}
(e^{\tau x_1}(\mathtt{tr}\circ G_\tau)^*(&e^{-\tau x_1}h)|qP_q(f))_{L^2(M)}\\
&=\<e^{\tau x_1}\mathtt{tr}\((\mathtt{tr}\circ G_\tau)^*(e^{-\tau x_1}h)\),(\Lambda_{g,q}-\Lambda_{g,0})f\>_{H^{-1/2,1/2}(\p M)}\\
&=\<h,e^{-\tau x_1}\(\mathtt{tr}\circ (\mathtt{tr}\circ G_\tau)^*)^*(e^{\tau x_1}(\Lambda_{g,q}-\Lambda_{g,0})f\)\>_{H^{-1/2,1/2}(\p M)}\\
&=\<h,S_\tau (\Lambda_{g,q}-\Lambda_{g,0})f\>_{H^{-1/2,1/2}(\p M)}.
\end{align*}
The proof is thus finished.
\end{proof}

The following result shows that the boundary integral equation is equivalent to the certain integral equation; compare with \cite[Proposition~3.2]{KSaU}.
\begin{Proposition}\label{solvability results are equivalent}
Suppose that $q\in L^\infty(M)$ and $0$ is not a Dirichlet eigenvalue of $-\Delta_g+q$ in $M$. There is $\tau_0>0$ such that for all $\tau\in\R$ with $|\tau|\ge\tau_0$ and for all $f,h\in\mathcal H_g(\p M)$, $\(\id+S_\tau(\Lambda_{g,q}-\Lambda_{g,0})\)h=f$ holds if and only if $\(\id+e^{-\tau x_1}G_\tau e^{\tau x_1}q\)P_q(h)=P_0(f)$.
\end{Proposition}
\begin{proof}
Suppose that $f,h\in\mathcal H_g(\p M)$ satisfies $\(\id+S_\tau(\Lambda_{g,q}-\Lambda_{g,0})\)h=f$. Note that by Theorem~\ref{existence of G_tau operator}, we can show that
$$
\Delta_g\(\id+e^{-\tau x_1}G_\tau e^{\tau x_1}q\)P_q(h)=qP_q(h)-qP_q(h)=0.
$$
Therefore, it is enough to prove that
$$
\mathtt{tr}\(\(\id+e^{-\tau x_1}G_\tau e^{\tau x_1}q\)P_q(h)\)=f,
$$
or equivalently
$$
h+\mathtt{tr}\(e^{-\tau x_1}G_\tau e^{\tau x_1}qP_q(h)\)=f.
$$
Using the factorization identity in Proposition~\ref{properties of S_tau(Lambda_q-Lambda_0)}, we can see that the left hand-side is $\(\id+S_\tau(\Lambda_{g,q}-\Lambda_{g,0})\)h$, which is equal to $f$ by assumption.
\medskip

The converse direction can be shown by applying $\mathtt{tr}$ to the both sides of the identity
$$
\(\id+e^{-\tau x_1}G_\tau e^{\tau x_1}q\)P_q(h)=P_0(f)
$$
and using the factorization identity in Proposition~\ref{properties of S_tau(Lambda_q-Lambda_0)}.
\end{proof}

\begin{Corollary}\label{corollary of solvability results are equivalent}
Suppose that $q\in L^\infty(M)$ and $0$ is not a Dirichlet eigenvalue of $-\Delta_g+q$ in $M$. There is $\tau_0>0$ such that for all $\tau\in\R$ with $|\tau|\ge\tau_0$, the operator $\id+S_\tau(\Lambda_{g,q}-\Lambda_{g,0}):\mathcal H_g(\p M)\to \mathcal H_g(\p M)$ is an isomorphism if and only if so is the operator $\id+e^{-\tau x_1}G_\tau e^{\tau x_1}q:b_q(M)\to b_0(M)$.
\end{Corollary}

The following proposition combined together with the above two results implies the solvability of the boundary integral equation \eqref{boundary integral equation}.
\begin{Proposition}\label{solvability of boundary integral equation}
Suppose that $q\in L^\infty(M)$ and $0$ is not a Dirichlet eigenvalue of $-\Delta_g+q$ in $M$. There is $\tau_0>0$ such that for all $\tau\in\R$ with $|\tau|\ge\tau_0$, the operator $\id+e^{-\tau x_1}G_\tau e^{\tau x_1}q:b_q(M)\to b_0(M)$ is an isomorphism.
\end{Proposition}
\begin{proof}
Since $\|G_\tau\|_{L^2(M)\to L^2(M)}\le \mathcal{O}(|\tau|^{-1})$ by Theorem~\ref{existence of G_tau operator}, the operator $\id+G_\tau q:L^2(M)\to L^2(M)$ is an isomorphism for big enough $|\tau|\gg 1$. Then for such $\tau$, the operator $\id+e^{-\tau x_1}G_\tau e^{\tau x_1}q:L^2(M)\to L^2(M)$ is an isomorphism whose inverse is $e^{-\tau x_1}(\id+G_\tau q)^{-1}e^{\tau x_1}$. Let $u\in b_0$ and $w=e^{-\tau x_1}(\id+G_\tau q)^{-1}e^{\tau x_1}u$. We need to show that $w\in b_q$. Applying $\id+e^{-\tau x_1}G_\tau e^{\tau x_1}q$ to $w$, we get that
$$
w+e^{-\tau x_1}G_\tau e^{\tau x_1}qw=u.
$$
Since $e^{\tau x_1}(-\Delta_g)e^{-\tau x_1}G_\tau=\id$ (by Theorem~\ref{existence of G_tau operator}), we get $(-\Delta_g)e^{-\tau x_1}G_\tau e^{\tau x_1}=\id$ and hence $(-\Delta_g+q)w=0$.
\end{proof}

\section{Complex geometrical optics solutions}\label{S7}
Let $q\in L^\infty(M)$ be such that $0$ is not a Dirichlet eigenvalue of $-\Delta_g+q$ in $M$, and let $\tau\in\R$ with $|\tau|\ge\tau_0$. In this section we construct the complex geometrical optics solutions for the Schr\"odinger equation $(-\Delta_g+q)u=0$ in $M$ whose trace is supported in $\Gamma_{\sgn(\tau)}$.

\subsection{\bf Solution operator}
To construct the complex geometrical optics solutions, we need to generalize Proposition~\ref{solvability result giving H_tau} to the case when the solution is determined on $S^-_{\tau}$.
\medskip

Set $\mathcal D=\{\psi\in C^\infty(M):\mathtt{tr}(\psi)=0\}$ and define
$$
M_\tau=\{(e^{-\tau x_1}(-\Delta_g)e^{\tau x_1}\psi,\mathtt{tr}_\nu(\psi)|_{S^+_{\tau}}):\psi\in \mathcal D\}\subset L^2(M)\times L^2(S^+_{\tau}).
$$

\begin{Lemma}\label{orthogonal complement to M_tau}
For $\tau\in\R$ with $|\tau|>0$, let $(u,u^+_{\tau})\in L^2(M)\times L^2(S^+_{\tau})$. Then $(u,u^+_{\tau})$ is in orthogonal to the closure of $M_\tau$ if and only if $e^{\tau x_1}(-\Delta_g)e^{-\tau x_1} u=0$, $\mathtt{tr}(u)|_{S^-_{\tau}}=0$ and $\mathtt{tr}(u)|_{S^+_{\tau}}=u^+_{\tau}$.
\end{Lemma}
\begin{proof}
Suppose that $(u,u^+_{\tau})$ is orthogonal to $M_\tau$. Then for $\psi\in \mathcal D$, we have
$$
(u|e^{-\tau x_1}(-\Delta_g)e^{\tau x_1}\psi)_{L^2(M)}+(u^+_{\tau}|\mathtt{tr}_\nu(\psi)|_{S^+_{\tau}})_{S^+_{\tau}}=0.
$$
Taking $\psi\in C^\infty_0(M^{\rm int})$, this gives $e^{\tau x_1}(-\Delta_g)e^{-\tau x_1} u=0$.
\medskip

Now, consider arbitrary $\psi\in \mathcal D$. Using the generalized Green's identity from Corollary~\ref{generalized Green's identity}, we get
$$
(u|e^{-\tau x_1}(-\Delta_g)e^{\tau x_1}\psi)_{L^2(M)}=-\<\mathtt{tr}(u),\mathtt{tr}_\nu(\psi)\>_{H^{-1/2,1/2}(\p M)}.
$$
Combining this together with the previous equality gives that $\mathtt{tr}(u)|_{S^-_{\tau}}=0$ and $\mathtt{tr}(u)|_{S^+_{\tau}}=u^+_{\tau}$.
\medskip

To prove the converse, suppose that $(u,u^+_{\tau})$ is such that $e^{\tau x_1}(-\Delta_g)e^{-\tau x_1} u=0$, $\mathtt{tr}(u)|_{S^-_{\tau}}=0$ and $\mathtt{tr}(u)|_{S^+_{\tau}}=u^+_{\tau}$. Then for $\psi\in \mathcal D$, we have
\begin{align*}
(u|e^{-\tau x_1}(-\Delta_g)e^{\tau x_1}&\psi)_{L^2(M)}\\
&=(e^{\tau x_1}(-\Delta_g)e^{-\tau x_1}u|\psi)_{L^2(M)}-\<\mathtt{tr}(u),\mathtt{tr}_\nu(\psi)\>_{H^{-1/2,1/2}(\p M)}\\
&=-\<\mathtt{tr}(u),\mathtt{tr}_\nu(\psi)\>_{H^{-1/2,1/2}(\p M)}\\
&=-(u^+_{\tau}|\mathtt{tr}_\nu (\psi)|_{S^+_{\tau}})_{S^+_{\tau}},
\end{align*}
which means that $(u,u^+_{\tau})$ is orthogonal to $M_\tau$.
\end{proof}
Let us denote by $m_\tau$ the operator of orthogonal projection onto the closure of $M_\tau$ in $L^2(M)\times L^2(S^+_{\tau})$.

\begin{Proposition}\label{solvability result giving R_tau}
Let $(M,g)$ be an admissible manifold. There are constants $C_0,\tau_0>0$ such that for all $\tau\in \R$ with $|\tau|\ge\tau_0$, $\delta>0$ and for given $f\in L^2(M)$ and $f^-_{\tau}\in L^2(S^-_{\tau})$, there exists a unique solution $u\in L^2(M)$ of the equation
$$
e^{\tau x_1}(-\Delta_g)e^{-\tau x_1} u=f\quad\text{in}\quad M
$$
such that $\mathtt{tr}(u)|_{S^-_\tau}=f^-_{\tau}$, $m_\tau(u,\mathtt{tr}(u)|_{S^+_{\tau}})=(u,\mathtt{tr}(u)|_{S^+_{\tau}})$
and
$$
\|u\|_{L^2(M)}\le C_0\frac{1}{|\tau|}\|f\|_{L^2(M)}+C_0\frac{1}{(\delta |\tau|)^{1/2}}\|f^-_{\tau}\|_{L^2(S^-_{\tau,\delta})}+C_0\|f^-_{\tau}\|_{L^2(S^0_{\tau,\delta})}.
$$
\end{Proposition}
\begin{proof}
Define a linear functional $l:M_\tau\to\C$ by
$$
l(e^{-\tau x_1}(-\Delta_g)e^{\tau x_1}\psi,\mathtt{tr}_\nu(\psi)|_{S^+_\tau})=(f|\psi)_{L^2(M)}-\(f^-_\tau|\mathtt{tr}_\nu(\psi)\)_{S^-_\tau}.
$$
On the orthogonal complement of $M_\tau$ we define $l$ to be zero. By the Carleman estimate~\eqref{Carleman estimate}, we have
\begin{align*}
|l(&e^{-\tau x_1}(-\Delta_g)e^{\tau x_1}\psi,\mathtt{tr}_\nu(\psi)|_{S^+_\tau})|\\
&\le \|f\|_{L^2(M)}\|\psi\|_{L^2(M)}+\|f^-_\tau\|_{S^-_{\tau,\delta}}\|\mathtt{tr}_\nu(\psi)\|_{S^-_{\tau,\delta}}+\|f^-_\tau\|_{S^0_{\tau,\delta}}\|\mathtt{tr}_\nu(\psi)\|_{S^0_{\tau,\delta}}\\
&\le C_0\(\frac{1}{|\tau|}\|f\|_{L^2(M)}+\frac{1}{(\delta |\tau|)^{1/2}}\|f^-_{\tau}\|_{L^2(S^-_{\tau,\delta})}+\|f^-_{\tau}\|_{L^2(S^0_{\tau,\delta})}\)\\
&\qquad\times \(\|e^{-\tau x_1}(-\Delta_g)e^{\tau x_1}\psi\|_{L^2(M)}+|\tau|^{1/2}\|\mathtt{tr}_\nu(\psi)\|_{S^+_\tau}\).
\end{align*}
By Riesz representation theorem, there is $(u,u^+_\tau)\in L^2(M)\times L^2(S^+_\tau)$ such that
$$
l(w,w^+_\tau)=(u|w)_{L^2(M)}+(u^+_\tau|w^+_\tau)_{S^+_\tau},
$$
for $(w,w^+_\tau)\in L^2(M)\times L^2(S^+_\tau)$. In particular, for $\psi\in\mathcal D$, we have
$$
(u|e^{-\tau x_1}(-\Delta_g)e^{\tau x_1}\psi)_{L^2(M)}+(u^+_\tau|\mathtt{tr}_\nu(\psi))_{S^+_\tau}=(f|\psi)_{L^2(M)}-\(f^-_\tau|\mathtt{tr}_\nu(\psi)\)_{S^-_\tau}.
$$
Taking $\psi\in C^\infty_0(M^{\rm int})$, gives that $e^{\tau x_1}(-\Delta_g)e^{-\tau x_1} u=f$. Moreover,
$$
\|u\|_{L^2(M)}\le C_0\frac{1}{|\tau|}\|f\|_{L^2(M)}+C_0\frac{1}{(\delta |\tau|)^{1/2}}\|f^-_{\tau}\|_{L^2(S^-_{\tau,\delta})}+C_0\|f^-_{\tau}\|_{L^2(S^0_{\tau,\delta})}.
$$
Since $l\equiv 0$ on the orthogonal complement of $M_{\tau}$ in $L^2(M)\times L^2(S^+_\tau)$, we have that $(u,\mathtt{tr}(u)|_{S^+_\tau})$ is in the closure of $M_{\tau}$ and hence $m_\tau(u,\mathtt{tr}(u)|_{S^+_\tau})=(u,\mathtt{tr}(u)|_{S^+_\tau})$.
\medskip

For arbitrary $\psi\in \mathcal D$, using the generalized Green's identity from Corollary~\ref{generalized Green's identity}, we get
$$
(u|e^{-\tau x_1}(-\Delta_g)e^{\tau x_1}\psi)_{L^2(M)}+\<\mathtt{tr}(u),\mathtt{tr}_\nu(\psi)\>_{H^{-1/2,1/2}(\p M)}=(f|\psi)_{L^2(M)}.
$$
Comparing this with the previous equality, this gives that $\mathtt{tr}(u)|_{S^-_\tau}=f^-_\tau$ and $\mathtt{tr}(u)|_{S^+_\tau}=u^+_\tau$.
\medskip

Now, we prove uniqueness. Suppose that $u'\in L^2(M)$ is another solution of the equation $e^{\tau x_1}(-\Delta_g)e^{-\tau x_1} u'=f$ satisfying all the conditions of the proposition. Then $e^{\tau x_1}(-\Delta_g)e^{-\tau x_1}(u-u')=0$, $\mathtt{tr}(u-u')|_{S^-_\tau}=0$, $\mathtt{tr}(u-u')|_{S^+_\tau}=u^+_\tau-u'{}^+_\tau$, and $(u-u',u^+_\tau-u'{}^+_\tau)$ is in the closure of $M_\tau$. However, by Lemma~\ref{orthogonal complement to M_tau}, $(u-u',u^+_\tau-u'{}^+_\tau)$ is orthogonal to the closure of $M_\tau$. Thus, we obtain $u-u'=0$ which finishes the proof.
\end{proof}

Let $R_\tau:L^2(M)\times L^2(S^-_\tau)\to L^2(M)$ be the solution operator obtained in the previous result. In other words, the operator $R_\tau$ is defined by $R_\tau (f,f^-_\tau)=u$, where $u,f,f^-_\tau$ are as in Proposition~\ref{solvability result giving R_tau}. 


\subsection{Construction of complex geometrical optics solutions}
Now, we are ready to construct complex geometrical optics solutions whose traces are supported in $\Gamma_{\sgn(\tau)}$. These are the solutions of the form
\begin{equation}\label{form of CGO}
u=e^{-\tau x_1}(a+r_0),
\end{equation}
where $r_0$ is a correction term and $a$ is an amplitude.

\subsubsection{\bf Construction in the case $q=0$}\label{CGO q=0} Recall that the transversal manifold $(M_0,g_0)$ is assumed to be simple. Let $\gamma:[0,T]\to M_0$ be the given geodesic in $(M_0,g_0)$. Choose another simple manifold $(\widetilde M_0,g_0)$ such that $(M_0,g_0)\subset\subset(\widetilde M_0^{\rm int},g_0)$ and extend the geodesic $\gamma$ in $\widetilde M_0$. Choose $\varepsilon>0$ such that $\gamma(t)\in \widetilde M_0\setminus M_0$ for all $t\in (-2\varepsilon,0)\cup (T,2\varepsilon)$ and set $p=\gamma(-\varepsilon)$ which is in $\widetilde M_0\setminus M_0$. Simplicity of $(\widetilde M_0,g_0)$ implies that there are globally defined polar coordinates $(r,\theta)$ centered at $p$. In these polar coordinates $\gamma$ corresponds to $r\mapsto (r,\theta_0)$ for some $\theta_0\in S^{n-2}$. Following \cite[Section~5.2]{DKSU}, we choose the following specific $a$:
$$
a(x_1,r,\theta)=e^{-i\tau r}|g|^{-1/4}c^{1/2}e^{i\lambda(x_1+ir)}b(\theta),
$$
where $\lambda\in \R$ and $b\in C^\infty(S^{n-2})$ is fixed such that $b$ is supported near $\theta_0$ so that $a=0$ near ${\p M_0\setminus E}$.
\medskip

Assume now that $u$ has the required form~\eqref{form of CGO}. Then the equation $(-\Delta_g)u_0=0$ is equivalent to
\begin{equation}\label{equation for r_0}
e^{\tau x_1}(-\Delta_g)e^{-\tau x_1}r_0=f,
\end{equation}
where $f:=e^{\tau x_1}\Delta_g e^{-\tau x_1}a$. Set $\Phi=x_1+ir$. Then a starightforward calculation shows that
\begin{align*}
f&=-e^{i\tau r}e^{\tau \Phi}(-\Delta_g)e^{-\tau\Phi}|g|^{-1/4}c^{1/2}e^{i\lambda(x_1+ir)}b(\theta)\\
&=-e^{i\tau r}[-\tau^2\<d\Phi,d\Phi\>_g+\tau(2\<d\Phi,d\cdot\>_g+\Delta_g\Phi)-\Delta_g](|g|^{-1/4}c^{1/2}e^{i\lambda(x_1+ir)}b(\theta)).
\end{align*}
Here the Riemannian inner product $\<\cdot,\cdot\>_g$ was extended as a complex bilinear form acting on complex valued $1$-forms. It was shown in \cite[Section~5]{DKSU} that $\<d\Phi,d\Phi\>_g=0$ and $(2\<d\Phi,d\cdot\>_g+\Delta_g\Phi)(|g|^{-1/4}e^{i\lambda(x_1+ir)}b(\theta))=0$. Hence, we get
\begin{equation}\label{def of f}
f=-e^{i\tau r}(-\Delta_{g})(|g|^{-1/4}c^{1/2}e^{i\lambda(x_1+ir)}b(\theta)).
\end{equation}
This shows that $\|f\|_{L^2(M)}\lesim 1$ as $\tau\to\infty$.
\medskip

We want to ensure that $\mathtt{tr}(u_0)$ is supported in $\Gamma_{\sgn(\tau)}$ where $\Gamma_{\sgn(\tau)}\supset \p M_{\sgn(\tau)}\cup \Gamma_{\rm a}$. To achieve this, following \cite{KSa}, we take a small parameter $\delta>0$ to be chosen later, and define the following sets
$$
V^{\tau,\delta}:=\{x\in S^-_{\tau}:\dist_{\p M}(x,\Gamma_{\rm i})<\delta\text{ or }x\in\p M_{\sgn(\tau)}\},\quad \Gamma_{\rm a}^{\tau,\delta}:=S^-_\tau\setminus V^{\tau,\delta}.
$$
Note that $\p M_{\sgn(\tau)}\cup \p M_{\rm tan}\subset (S^-_\tau)^{\rm int}$. For the boundary condition, we set
$$
f_{\tau,\delta}^-:=\begin{cases}
-a&\text{ on }\hfill V^{\tau,\delta},\\
0&\text{ on }\hfill\Gamma_{\rm a}^{\tau,\delta}.
\end{cases}
$$
Defining $f_{\tau,\delta}^-$ in such a way, we have $f_{\tau,\delta}^-|_{\Gamma_{\rm a}^{\tau,\delta}\cap \p M_{\rm tan}}=0$. Recall that $\Gamma_{\rm i}\subset \R\times(\p M_0\setminus E)$ and $a$ was chosen in a way to satisfy $a=0$ near ${\p M_0\setminus E}$. Therefore, $f_{\tau,\delta}^-|_{V^{\tau,\delta}\cap \p M_{\rm tan}}=0$, and hence we have
$$
f_{\tau,\delta}^-|_{\p M_{\rm tan}}=0.
$$
Since $\|f_{\tau,\delta}^-\|_{L^\infty(S^-_\tau)}\lesim 1$, we obtain the following estimates
$$
\|f^-_{\tau,\delta}\|_{L^2(S^-_{\tau,\delta})}\lesim \sigma_{\p M}(S^-_{\tau,\delta})
$$
and
\begin{multline*}
\|f^-_{\tau,\delta}\|_{L^2(S^0_{\tau,\delta})}\lesim \sigma_{\p M}\(\{x\in \p M:-\delta<\sgn(\tau)\p_\nu \varphi(x)<0\}\)\\
+\sigma_{\p M}\(\{x\in \p M:0<\sgn(\tau)\p_\nu \varphi(x)<(3|\tau|)^{-1}\}\).
\end{multline*}
If we set
$$
r_0=R_\tau(f,f_{\tau,\delta}^-),
$$
then, by Proposition~\ref{solvability result giving R_tau}, $r_0$ solves \eqref{equation for r_0} with $\mathtt{tr}(r_0)|_{S^-_\tau}=f^-_{\tau,\delta}$ and satisfies
\begin{multline*}
\|r_0\|_{L^2(M)}\lesim \frac{1}{|\tau|}+\frac{1}{(\delta |\tau|)^{1/2}}\sigma_{\p M}(S^-_{\tau,\delta})\\
+\sigma_{\p M}\(\{x\in \p M:-\delta<\sgn(\tau)\p_\nu \varphi(x)<0\}\)\\
+\sigma_{\p M}\(\{x\in \p M:0<\sgn(\tau)\p_\nu \varphi(x)<(3|\tau|)^{-1}\}\).
\end{multline*}
Thus, there is constant $C_0>0$ such that
$$
\|r_0\|_{L^2(M)}\le C_0\(\frac{1}{|\tau|}+\frac{1}{(\delta |\tau|)^{1/2}}+o_{\tau\to\infty}(1)+o_{\delta\to 0}(1)\).
$$
We choose $\delta$ such that $C_0o_{\delta\to 0}(1)\le \varepsilon/2$. Then we take $|\tau|\ge\tau_0$ large enough so that
$$
C_0\(\frac{1}{|\tau|}+\frac{1}{(\delta |\tau|)^{1/2}}+o_{\tau\to\infty}(1)\)\le \varepsilon/2.
$$
Therefore, we get $\|r_0\|_{L^2(M)}\to 0$ as $\tau\to\infty$. This will give the complex geometrical optics solution $u_0=e^{-\tau x_1}(a+r_0)$ to $(-\Delta_g)u_0=0$ whose trace is supported in $\Gamma_{\sgn(\tau)}$. Thus, we have proved the following proposition.

\begin{Proposition}\label{CGO harmonic}
Let $(M,g)$ be an admissible manifold. Suppose that $\gamma:[0,T]\to M_0$ is a given nontangential geodesic in $(M_0,g_0)$, and let $\theta_0\in S^{n-2}$ be as in the begining of Section~\ref{CGO q=0}. For $\tau\in\R$ with $|\tau|\ge\tau_0$ and $\delta>0$, for any $\lambda\in \R$ and for any $b\in C^\infty(S^{n-2})$ supported sufficiently close to $\theta_0$, there is a solution $u_0\in H_{\Delta_g}(M)$ to the equation $(-\Delta_g)u_0=0$ of the form
$$
u_0=e^{-\tau x_1}(a+r_0),
$$
and satisfying
$$
\supp(\mathtt{tr}(u_0))\subset \Gamma_{\sgn(\tau)}
$$
where
$$
a=e^{-i\tau r}|g|^{-1/4}c^{1/2}e^{i\lambda(x_1+ir)}b(\theta),
$$
and $\|r_0\|_{L^2(M)}\to 0$ as $\tau\to\infty$.
\end{Proposition}

\begin{Remark}\label{CGO harmonic supported in pM_(sgn(tau))}{\rm
Modifying the above arguments in appropriate places, one can construct complex geometrical optics solutions whose traces are supported in $\p M_{\sgn(\tau)}$ if $\p M_{\rm tan}$ has zero measure in $\p M$. Let us indicate these modifications. Up to \eqref{def of f} everything is same except that we do not put any restrictions on $b$, so that we do not require $a$ to vanish on any part of the boundary. In order to ensure that $\supp(\mathtt{tr}(u))\subset\p M_{\sgn(\tau)}$, for fixed $\delta>0$, we set
$$
f^-_{\tau,\delta}:=-a.
$$
Since $\|f_{\tau,\delta}^-\|_{L^\infty(S^-_\tau)}\lesim 1$ and $\sigma_{\p M}(\p M_{\rm tan})=0$, we obtain the following estimates
$$
\|f^-_{\tau,\delta}\|_{L^2(S^-_{\tau,\delta})}\lesim \sigma_{\p M}(S^-_{\tau,\delta})
$$
and
\begin{multline*}
\|f^-_{\tau,\delta}\|_{L^2(S^0_{\tau,\delta})}\lesim \sigma_{\p M}\(\{x\in \p M:-\delta<\sgn(\tau)\p_\nu \varphi(x)<0\}\)\\
+\sigma_{\p M}\(\{x\in \p M:0<\sgn(\tau)\p_\nu \varphi(x)<(3|\tau|)^{-1}\}\).
\end{multline*}
We use Proposition~\ref{solvability result giving R_tau} to solve \eqref{equation for r_0} for $r_0$ with $\mathtt{tr}(r_0)|_{S^-_\tau}=f^-_{\tau,\delta}$ and to show that $r_0$ satisfies the same estimate as before for some $C_0>0$ constant:
$$
\|r_0\|_{L^2(M)}\le C_0\(\frac{1}{|\tau|}+\frac{1}{(\delta |\tau|)^{1/2}}+o_{\tau\to\infty}(1)+o_{\delta\to 0}(1)\).
$$
Thus, we have constructed the complex geometrical optics solution $u_0\in H_{\Delta_g}(M)$ to $(-\Delta_g)u_0=0$ of the form
$$
u_0=e^{-\tau x_1}(a+r_0)
$$
whose trace is supported in $\p M_{\sgn(\tau)}$ and $\|r_0\|_{L^2(M)}\to 0$ as $\tau\to\infty$.
}\end{Remark}

\subsubsection{\bf Construction for general $q$} Next, we construct complex geometrical optics solutions for the Schr\"odinger equation $(-\Delta_g+q)u=0$ in $M$ with $q\in L^\infty(M)$ such that $\supp(\mathtt{tr}(u))\subset\Gamma_{\sgn(\tau)}$.

\begin{Proposition}\label{CGO Schrodinger}
Let $(M,g)$ be an admissible manifold and let $q\in L^\infty(M)$ be such that $0$ is not a Dirichlet eigenvalue of $-\Delta_g+q$ in $M$. Suppose that $\gamma:[0,T]\to M_0$ is a given nontangential geodesic in $(M_0,g_0)$. For any $\tau\in\R$ with $|\tau|\ge \tau_0$, there is a solution $u\in L^2(M)$ to the equation $(-\Delta_g+q)u=0$ of the form
$$
u=u_0+e^{-\tau x_1}r_1,
$$
where $u_0$ is as in Proposition~\ref{CGO harmonic} and one has
$$
(\id+G_\tau\circ q)r_1=-G_\tau q e^{\tau x_1}u_0,
$$
and $\|r_1\|_{L^2(M)}\lesim \frac{1}{|\tau|}$ as $\tau\to\infty$. Moreover, $\mathtt{tr}(u)$ is supported in $\Gamma_{\sgn(\tau)}$.
\end{Proposition}
\begin{proof}
Consider the following integral equation
\begin{equation}\label{integral equation}
(\id+G_\tau \circ q)r_1=-G_\tau qe^{\tau x_1}u_0.
\end{equation}
Since $q\in L^\infty(M)$ and $\|G_\tau\|_{L^2(M)\to L^2(M)}\lesim \frac{1}{|\tau|}$ by Theorem~\ref{existence of G_tau operator}, for sufficiently large this integral equation has a unique solution $r_1=-(\id+G_\tau \circ q)^{-1}G_\tau qe^{\tau x_1}u_0$ in terms of the convergent Neumann series. Then $\|G_\tau\|_{L^2(M)\to L^2(M)}\lesim \frac{1}{|\tau|}$ implies that $\|r_1\|_{L^2(M)}\lesim \frac{1}{|\tau|}$. Using the fact that $(-\Delta_g)e^{-\tau x_1}G_\tau=e^{-\tau x_1}$ (by Theorem~\ref{existence of G_tau operator}) and that $(-\Delta_g)u_0=0$, and using \eqref{integral equation}, we can show
\begin{align*}
(-\Delta_g+q)u&=(-\Delta_g+q)u_0+(-\Delta_g+q)e^{-\tau x_1}r_1\\
&=qu_0+(-\Delta_g+q)(-e^{-\tau x_1}G_\tau qr_1-e^{-\tau x_1}G_\tau qe^{\tau x_1}u_0)\\
&=qu_0-e^{-\tau x_1}qr_1-qu_0-qe^{-\tau x_1}G_\tau qr_1-qe^{-\tau x_1}G_\tau qe^{\tau x_1}u_0\\
&=-e^{-\tau x_1}q(\id+G_\tau \circ q)r_1-qe^{-\tau x_1}G_\tau qe^{\tau x_1}u_0\\
&=e^{-\tau x_1}qG_\tau q e^{\tau x_1}u_0-qe^{-\tau x_1}G_\tau qe^{\tau x_1}u_0\\
&=0.
\end{align*}

Let us now prove the last part of the proposition. By Proposition~\ref{CGO harmonic}, we have that $\mathtt{tr}(u_0)$ is supported in $\Gamma_{\sgn(\tau)}$. Note that Theorem~\ref{existence of G_tau operator} implies that $\mathtt{tr}(G_\tau q r_1)$ is supported in $S^+_\tau$. These, together with \eqref{integral equation} imply that the trace of $u=u_0+e^{-\tau x_1}r_1$ is supported in $\Gamma_{\sgn(\tau)}$.
\end{proof}

\begin{Remark}\label{CGO Schrodinger supported in pM_(sgn(tau))}{\rm
If $\p M_{\rm tan}$ has zero measure in $\p M$, one can replace $u_0$ in the above proposition with the one obtained in Remark~\ref{CGO harmonic supported in pM_(sgn(tau))}. Then the proof of Proposition~\ref{CGO Schrodinger} shows that so-obtained complex geometrical optics solution $u\in H_{\Delta_g}(M)$ to $(-\Delta_g+q)u_0=0$ of the form
$$
u=u_0+e^{-\tau x_1}r_1
$$
has $\supp(\mathtt{tr}(u))\subset \p M_{\sgn(\tau)}$ and $\|r_1\|_{L^2(M)}\lesim \frac{1}{|\tau|}$ as $\tau\to\infty$.
}\end{Remark}

\section{Proofs of the main results}\label{S8}

\begin{proof}[Proof of Theorem~\ref{main th}]Suppose that $q\in C(M)$ such that $0$ is not a Dirichlet eigenvalue of $-\Delta_g+q$.  Assume the knowledge of $(M,g)$ and $\Lambda_{g,q}f$ on $\Gamma_-$ for all $f\in \mathcal H(\p M)$ supported in $\Gamma_+$. Then by Proposition~\ref{extending difference of DN maps}, the following integral identity holds
\begin{equation}\label{rewritten integral identity}
\<\mathtt{tr}(u_2),(\Lambda_{g,q}-\Lambda_{g,0})\mathtt{tr}(u_1)\>_{H^{-1/2,1/2}(\p M)}=(u_2|qu_1)_{L^2(M)},
\end{equation}
where $u_1\in H_{\Delta_g}(M)$ is a solution of $(-\Delta_g+q)u_1=0$ in $M$ with $\mathtt{tr}(u_1)$ supported in $\Gamma_+$, and $u_2\in H_{\Delta_g}(M)$ is a solution of $(-\Delta_g)u_2=0$ in $M$ with $\mathtt{tr}(u_2)$ supported in $\Gamma_-$.
\medskip

Let $\tau\ge\tau_0$. By Proposition~\ref{CGO Schrodinger}, there is $u_1\in H_{\Delta_{g}}(M)$ solving $(-\Delta_{g}+q)u_1=0$ in $M$ with $\mathtt{tr}(u_1)$ supported in $\Gamma_+$, and having the form
$$
u_1=e^{-\tau x_1}(e^{-i\tau r}|g|^{-1/4}c^{1/2}e^{i\lambda(x_1+ir)}b(\theta)+r'+r_0)=u_1'+e^{-\tau x_1}r_0
$$
where $\|r_0\|_{L^2(M)}\lesim \frac{1}{|\tau|}$ and $\|r'\|_{L^2(M)}\to 0$ as $\tau\to+\infty$ (here $u'_1$ is a solution to $(-\Delta_{g})u_1'=0$ as in Proposition~\ref{CGO harmonic}).
\medskip

By Proposition~\ref{CGO harmonic}, there is a $\overline{u}_2\in H_{\Delta_{g}}(M)$ solving $(-\Delta_g)\overline{u}_2=0$ in $M$ with $\mathtt{tr}({\overline u}_2)$ supported in $\Gamma_-$, and having the form
$$
\overline u_2=e^{\tau x_1}(e^{i\tau r}|g|^{-1/4}c^{1/2}e^{i\lambda(x_1+ir)}+r'')
$$
where $\|r''\|_{L^2(M)}\to 0$ as $\tau\to+\infty$. Then $u_2$ will be the complex geometrical optics solutions to $(-\Delta_{g}) u_2=0$ in $M$ with $\mathtt{tr}(u_2)$ supported in $\Gamma_-$, and having the form
$$
u_2=e^{\tau x_1}(e^{-i\tau r}|g|^{-1/4}c^{1/2}e^{-i\lambda(x_1-ir)}+\overline{r''})
$$

The important thing to note is that $u'_1$ as well as $u_2$ depend only on $(M,g)$, i.e. independent on $q$. Since $\mathtt{tr}(u_1)$ is supported in $\Gamma_+$ and $\mathtt{tr}(u_2)$ is supported in $\Gamma_-$, the left hand-side of \eqref{rewritten integral identity} requires only the given partial data of $\Lambda_{g,q}$.
\medskip

Now, we show that $\mathtt{tr}(u_1)$ can be reconstructed from the above mentioned partial knowledge of $\Lambda_{g,q}$. By Proposition~\ref{CGO Schrodinger} and Proposition~\ref{solvability results are equivalent}, one can check that $\mathtt{tr}(u_1)$ satisfies the following boundary integral equation
$$
(\id+S_\tau(\Lambda_{g,q}-\Lambda_{g,0}))\mathtt{tr}(u_1)=\mathtt{tr}(u_1').
$$
Then Corollary~\ref{corollary of solvability results are equivalent} and Proposition~\ref{solvability of boundary integral equation} imply solvability of the above boundary integral equation for sufficiently large $\tau$. Substituting this solution $\mathtt{tr}(u_1)$ into the left hand-side of \eqref{rewritten integral identity}, we can determine
$$
(u_2|qu_1)_{L^2(M)}
$$
for all complex geometrical optics solutions $u_1,u_2$ of the above form.
\medskip

Using the decay properties of $r',r'',\tilde r$ and taking limit as $\tau\to \infty$, we can reconstruct
$$
\int_M cq|g|^{-1/2}e^{2i\lambda(x_1+ir)}b(\theta)\,d\Vol_g.
$$
Now we extend $q$ as zero to $\R\times M_0$. Since $d\Vol_g=|g|^{1/2}dx_1\,dr\,d\theta$, the above expression becomes
$$
\int_{S^{n-2}}\int_0^\infty e^{-2\lambda r}\(\int_{-\infty}^\infty e^{2i\lambda x_1}(cq)(x_1,r,\theta)\,dx_1\)\,dr\,d\theta.
$$
Varying $b\in C^\infty(S^{n-2})$ so that the support of $b$ is sufficiently close to $\theta_0$ and noting that the term in the brackets is the one-dimensional Fourier transform of $q$ with respect to the $x_1$-variable, which we denote by $\widehat{q}$, we determine
$$
\int_0^\infty e^{-2\lambda r}\,\widehat{(cq)}(2\lambda,r,\theta_0)\,dr.
$$
Recalling that $r\mapsto (r,\theta_0)$ corresponds to the given nontangential geodesic $\gamma:[0,T]\to M$, we finish the proof.
\end{proof}

\begin{proof}[Proof of Theorem~\ref{main th2}]
Assume that $O\subset M_0$ is open such that $O\cap \p M_0\subset E$ and the local geodesic ray transform is invertible on $O$. According to Theorem~\ref{main th}, we can constructively determine
\begin{equation}\label{attenuated ray transform}
\int_0^T e^{-2\lambda t}\widehat{(cq)}(2\lambda,\gamma(t))\,dt
\end{equation}
for all nontangential geodesics $\gamma:[0,T]\to O$ with $\gamma(0),\gamma(T)\in E$. This is the local attenuated geodesic ray transform of $\widehat{(cq)}(2\lambda,\cdot)$ in $O$, with attenuation $-2\lambda$. Setting $\lambda=0$, we determine an unattenuated local geodesic ray transform of $\widehat{(cq)}(0,\cdot)$ in $O$. Then using the constructive invertibility assumption for the local geodesic ray transform, we recover $\widehat{(cq)}(0,\cdot)$ in $O$.
\medskip

Now, we go back to \eqref{attenuated ray transform} and differentiate it with respect to $\lambda$ at $\lambda=0$. Since we have reconstructed $\widehat{(cq)}(0,\cdot)$, we constructively determine the local geodesic ray transform of $\(\de{}{\lambda}\widehat{(cq)}\)(0,\cdot)$ in $O$. Using the invertibility assumption for the local geodesic ray transform again, we obtain $\(\de{}{\lambda}\widehat{(cq)}\)(0,\cdot)$ in $O$.
\medskip

Using this argument iteratively by taking higher derivatives of \eqref{attenuated ray transform} with respect to $\lambda$, we can reconstruct
$$
\(\de{{}^k}{\lambda^k}\widehat{(cq)}\)(0,\cdot)\,\text{ in }\,O\,\text{ for all integers }\,k\ge 0.
$$
Since $q$ is compactly supported in $x_1$-variable, its Fourier transform $\widehat{(cq)}(\lambda,\cdot)$ is analytic with respect to $\lambda$. Therefore, we have reconstructed the Taylor series expansion of $\widehat{(cq)}(\lambda,\cdot)$ in $O$. Then we determine $q$ in $M\cap(\R\times O)$ by inverting the one-dimensional Fourier transform of $cq$ with respect to the $x_1$-variable.
\end{proof}

\begin{proof}[Proof of Theorem~\ref{main th3}]
Let $(M,g)$ be a known admissible manifold such that $\p M_{\rm tan}$ is of measure zero in $\p M$. Suppose that $q\in C(M)$ such that $0$ is not a Dirichlet eigenvalue of $-\Delta_g+q$.  Assume the knowledge of $\Lambda_{g,q}f$ on $\p M_-$ for all $f\in \mathcal H(\p M)$ supported in $\p M_+$.
\medskip

Using Remark~\ref{CGO harmonic supported in pM_(sgn(tau))} and Remark~\ref{CGO Schrodinger supported in pM_(sgn(tau))}, as in the proof of Theorem~\ref{main th}, we can construct $u_1\in H_{\Delta_{g}}(M)$ and $u_2\in H_{\Delta_{g}}(M)$ solving $(-\Delta_{g}+q)u_1=0$ in $M$ with $\mathtt{tr}(u_1)$ supported in $\p M_+$ and solving $(-\Delta_{g}) u_2=0$ in $M$ with $\mathtt{tr}(u_2)$ supported in $\p M_-$, respectively, and having the forms
\begin{align*}
u_1&=e^{-\tau x_1}(e^{-i\tau r}|g|^{-1/4}c^{1/2}e^{i\lambda(x_1+ir)}b(\theta)+r'+r_0)=u_1'+e^{-\tau x_1}r_0,\\
u_2&=e^{\tau x_1}(e^{-i\tau r}|g|^{-1/4}c^{1/2}e^{-i\lambda(x_1-ir)}+\overline{r''}),
\end{align*}
where $\|r_0\|_{L^2(M)}\lesim \frac{1}{|\tau|}$, $\|r'\|_{L^2(M)}\to 0$ and $\|r''\|_{L^2(M)}\to 0$ as $\tau\to+\infty$ (here $u'_1$ is a solution to $(-\Delta_{g})u_1'=0$ as in Remark~\ref{CGO harmonic supported in pM_(sgn(tau))}).
\medskip

Continuing as in the proof of Theorem~\ref{main th}, but replacing $\Gamma_\pm$ by $\p M_\pm$, for any $\lambda\in\R$, we can constructively determine
$$
\int_0^T e^{-2\lambda t}\widehat{(cq)}(2\lambda,\gamma(t))\,dt
$$
for all the nontangential geodesics $\gamma:[0,T]\to M_0$ in $(M_0,g_0)$. Using the constructive invertibility assumption for the global geodesic ray transform, we reconstruct $q$ in $M$ via similar steps as in the proof of Theorem~\ref{main th2}.
\end{proof}

\section*{Acknowledgements}
The author thanks his advisor Professor Gunther Uhlmann for all his encouragement and support. Many thanks to Professor Mikko Salo 
very useful comments. The work of the author was partially supported by the National Science Foundation.


\begin{thebibliography}{ABC}
\bibitem{BU} A. Bukhgeim, G. Uhlmann, \emph{Recovering a potential from partial Cauchy data}, Comm. Partial Diff. Eq. {\bf 27} (2002), no. 3-4, 653--668.

\bibitem{Cal} A. Calder\'on, \emph{On an inverse boundary value problem}, Seminar on Numerical Analysis and its Applications to Continuum Physics, 65--73, Soc. Brasil. Mat., Rio de Janeiro, 1980.

\bibitem{DKSU} D. Dos Santos Ferreira, C. Kenig, M. Salo, G. Uhlmann, \emph{Limiting Carleman weights and anisotropic inverse problems}, Invent. Math. {\bf 178} (2009), no. 1, 119--171.

\bibitem{DKS} D. Dos Santos Ferreira, C. Kenig, M. Salo, \emph{Determining an unbounded potential from Cauchy data in admissible geometries}, Comm. PDE {\bf 38} (2013), no. 1, 50--68.

\bibitem{GTr} D. Gilbarg, N. S. Trudinger, \emph{Elliptic partial differential equations of second order}, Third printing, Springer-Verlag, 200.

\bibitem{Is} V. Isakov, \emph{On uniqueness in the inverse conductivity problem with local data}, Inverse Probl. Imaging {\bf 1} (2007), no. 1, 95--105.

\bibitem{KSa-survey} C. Kenig, M. Salo, \emph{Recent progress in the Calder\'on problem with partial data}, Contemp. Math. {\bf 615} (2014), 193--222.

\bibitem{KSa} C. Kenig, M. Salo, \emph{The Calder\'on problem with partial data on manifolds and applications}, Analysis \& PDE {\bf 6} (2013), no. 8, 2003--2048.

\bibitem{KSaU} C. Kenig, M. Salo, G. Uhlmann, \emph{Reconstructions from boundary measurements on admissible manifolds}, Inverse Probl. Imaging {\bf 5} (2011), no. 4, 859--877.

\bibitem{KSU} C. Kenig, J. Sj\"ostrand, G. Uhlmann, \emph{The Calder\'on problem with partial data}, Ann. of Math. (2) {\bf 165} (2007), no. 2, 567--591.

\bibitem{Kr} V. Krishnan, \emph{A generalization of inversion formulas of Pestov and Uhlmann}, J. Inv. Ill-Posed Problems, {\bf 18} (2010), 401--408.

\bibitem{LM}  J.-L. Lions and E. Magenes, \emph{Probl\`emes aux limites non homog\`enes et applications. Vol. 1}, Travaux et Recherches Math\'ematiques, No. 17, Dunod, Paris, 1968. MR MR0247243 (40 \#512).

\bibitem{LU} J. Lee, G. Uhlmann, \emph{Determining anisotropic real-analtic conductivities by boundary measurements}, Comm. Pure Appl. Math., {\bf 42} (1989), no. 8, 1097--1112.

\bibitem{Na} A. Nachman, \emph{Reconstructions from boundary measurements}, Ann. Math. {\bf 128} (1988), 531--576.

\bibitem{NaS} A. Nachman, B. Street, \emph{Reconstruction in the Calder\'on Problem with Partial Data}, Comm. PDE {\bf 35} (2010), 375--390.

\bibitem{No} R.G. Novikov, \emph{Multidimensional inverse spectral problem for the equation $-\Delta\psi+(v(x)-Eu(x))\psi=0$}, Funct. Anal. Appl. {\bf 22} (1988), 263--272.

\bibitem{PeU} L. Pestov, G. Uhlmann, {\it On the Characterization of the Range and Inversion Formulas for the Geodesic X-Ray Transform}, International Math. Research Notices, \textbf{80} (2004), 4331--4347.

\bibitem{SaU} M. Salo, G. Uhlmann, \emph{The attenuated ray transform on simple surfaces}, J. Diff. Geom. {\bf 88} (2011), no. 1, 161--187.

\bibitem{Shar} V. A. Sharafutdinov, \emph{Integral geometry of tensor fields}, Inverse and Ill-Posed Problems Series, VSP, Utrecht, 1994.

\bibitem{SyU} J. Sylvester, G. Uhlmann, \emph{ A global uniqueness theorem for an inverse boundary value problem}, Ann. of Math. (2) {\bf 125} (1987), no. 1, 153--169.

\bibitem{Tay} M. E. Taylor, \emph{Partial Differential Equations III. Nonlinear Equations}, Applied Mathematical Sciences, 117. Springer, New York, 2011.

\bibitem{U} G. Uhlmann, \emph{Inverse problems: seeing the unseen}, Bull. Math. Sci. {\bf 4} (2014), no. 2, 209--279.

\bibitem{UV} G. Uhlmann, A. Vasy, \emph{The inverse problem for the local geodesic ray transform, Inventiones mathematicae}, DOI: 10.1007/s00222-015-0631-7.

\end{thebibliography}
\end{document}